\documentclass[sort&compress,preprint,11pt]{elsarticle}

\usepackage{amssymb}
\usepackage{amsmath,amsthm}
\usepackage{lineno,hyperref}
\usepackage{amsfonts}
\usepackage{amsmath}
\usepackage{rotating}
\usepackage{amssymb}
\usepackage{amscd,multirow}
\usepackage{graphicx}
\usepackage{epstopdf}
\usepackage[titletoc]{appendix}
\usepackage[justification=centering]{caption}
\usepackage{subeqnarray,cases}
\usepackage{braket,amsfonts}
\usepackage{cleveref}
\usepackage{mathtools}
\usepackage{array}
\usepackage{algpseudocode}%
\usepackage{listings}%
\usepackage[misc]{ifsym}
\usepackage[justification=centering]{caption}
\usepackage{multirow}%

\usepackage{caption}
\usepackage{subcaption}
\usepackage{subcaption,natbib}
\usepackage{pgfplots}
 \usepackage{float}
\usepackage{makecell}
\usepackage{algorithm}

\newtheorem{theorem}{Theorem}[section]

\newtheorem{lemma}{Lemma}[section]
\newtheorem{proposition}{Proposition}[section]
\newtheorem{remark}{Remark}[section]

\newtheorem{assumption}{Assumption}[section]
\allowdisplaybreaks[4]

\modulolinenumbers[5]

\usepackage[a4paper,margin=1in]{geometry}
\usepackage{amsmath,amssymb,amsthm,mathtools}
\usepackage{bm}
\usepackage{enumitem}

\newcommand{\R}{\mathbb{R}}
\newcommand{\ip}[2]{\left\langle #1,#2\right\rangle}

\newcommand{\bigO}{\mathcal{O}}

      \renewcommand{\thesection}{\arabic{section}}

\begin{document}

\begin{frontmatter}



\title{Convergence of iterates and improved rates for accelerated augmented Lagrangian methods for linearly constrained convex optimization
}


\author{Xin He} 

\affiliation{organization={School of Science, Xihua University},
            city={Chengdu},
            postcode={610039}, 
            state={Sichuan},
            country={China}}

\ead{hexinuser@163.com}

\author{Nan-Jing Huang}
\affiliation{organization={Department of Mathematics, Sichuan University},
            city={Chengdu},
            postcode={610064}, 
            state={Sichuan},
            country={China}}

\ead{njhuang@scu.edu.cn}

\author{Yi-Bin Xiao}
\affiliation{organization={Department of Mathematics, University of Electronic Science and Technology of China},
            city={Chengdu},
            postcode={611731}, 
            state={Sichuan},
            country={China}}
\ead{xiaoyb9999@hotmail.com}

\author{Ya-Ping Fang}
\affiliation{organization={Department of Mathematics, Sichuan University},
            city={Chengdu},
            postcode={610064}, 
            state={Sichuan},
            country={China}}
\ead{ypfang@scu.edu.cn}

\begin{abstract}
Motivated by an inertial primal-dual dynamical system with vanishing damping, we propose a class of accelerated augmented
Lagrangian methods with Nesterov extrapolation parameters for a linearly constrained convex optimization problem with a
differentiable objective function. The framework
contains two variants: an implicit-gradient scheme for convex continuously
differentiable objectives and a partially explicit scheme for convex smooth
objectives. Under suitable parameter conditions, we prove convergence of the
 primal-dual sequence to a primal-dual solution, together with accelerated estimates for the augmented Lagrangian gap, the feasibility violation, and the objective
residual. In the noncritical parameter regime, these estimates are improved from
$\bigO(1/k^2)$ to $o(1/k^2)$. Numerical experiments are also presented to illustrate the theoretical results. To the best of our knowledge, neither $o(1/k^2)$ rates for both feasibility violation and objective residual nor  convergence of iterates under the critical parameter condition have  been previously established for accelerated augmented Lagrangian-type methods in this setting.
\end{abstract}


\begin{keyword}
 Accelerated augmented Lagrangian method; linearly constrained convex optimization; Nesterov acceleration; iterate convergence; $o(1/k^2)$ rate

\vspace{0.5em}


\end{keyword}

\end{frontmatter}

 \section{Introduction}
 
 \subsection{Problem formulation}

We consider the linearly constrained convex optimization problem
\begin{equation}\label{prob_main}
    \min_{x\in\mathbb R^n} f(x),
    \qquad \text{\rm s.t. } Ax=b,
\end{equation}
where $f:\mathbb R^n\to\mathbb R$ is convex, $A\in\mathbb R^{m\times n}$,
and $b\in\mathbb R^m$. Problem \eqref{prob_main} is a basic model in
constrained convex optimization and arises in a wide range of applications,
including signal and image reconstruction, statistical learning, distributed
optimization, resource allocation, and network flow problems; see
\cite{BoydADMM,LinLF,ZhaoJMLR,ZengSiam,BaiJSC,LuoMc}.

For a penalty parameter $\beta\ge0$, the augmented Lagrangian associated with
\eqref{prob_main} is defined by
\[
    \mathcal L_\beta(x,\lambda)
    =
    f(x)+\langle \lambda,Ax-b\rangle
    +\frac{\beta}{2}\|Ax-b\|^2,
    \qquad
    (x,\lambda)\in\mathbb R^n\times\mathbb R^m .
\]
The quadratic penalty term improves the control of feasibility, while the
multiplier variable preserves the primal-dual structure of the original
constrained problem. This paper studies accelerated augmented Lagrangian method (ALM)
for \eqref{prob_main}. The acceleration mechanism is related to Nesterov
extrapolation rules \cite{Nesterov1983,Nesterov2004}, but the presence
of the linear constraint requires a coupled primal-dual treatment. We distinguish
two objective function settings.

The first setting assumes that $f$ is convex and continuously differentiable,
without imposing Lipschitz continuity of $\nabla f$. This class is relevant
because convex differentiable objectives are not necessarily globally smooth. For
example, the $\ell_p$-loss data-fitting objective, also known as the
least-$p$-powers regression objective,
\[
    f(x)=\frac1p\|Bx-c\|_p^p,\qquad 1<p<2,
\]
is convex and continuously differentiable, but its gradient is not globally
Lipschitz near points satisfying $Bx-c=0$. Such nonquadratic fidelity terms
arise in robust regression, inverse problems, and random sampling; see \cite{Woodruff,Ghadiri2021,AdilP,ShenSiam}. Power convex
functions with superquadratic growth also provide natural examples whose
gradients are not globally Lipschitz on $\mathbb R^n$. Hence it is meaningful to
develop accelerated augmented Lagrangian schemes beyond the standard
$L$-smooth framework. The second setting assumes that $f$ has an $L$-Lipschitz continuous gradient.
This is the classical smooth convex setting in first-order optimization
\cite{LinLF,BeckSiam,Nesterov1983,Nesterov2004}. Under this additional
smoothness assumption, the objective term can be linearized at an extrapolated
point, which leads to a partially explicit algorithm and allows us to derive
additional estimates for gradient and stationarity residuals. Thus the two cases
considered in this paper cover both the general convex $C^1$ setting and the
smooth convex setting.

Throughout the paper, we assume that the KKT set $\Omega$ is nonempty, where
\[
    \Omega
    :=
    \left\{
    (x^*,\lambda^*)\in\mathbb R^n\times\mathbb R^m:
    Ax^*=b,\quad
    \nabla f(x^*)+A^\top\lambda^*=0
    \right\}.
\]
Equivalently, $(x^*,\lambda^*)\in\Omega$ is a saddle point of the augmented
Lagrangian, namely
\begin{equation}\label{eq:saddle}
    \mathcal L_\beta(x^*,\lambda)
    \le
    \mathcal L_\beta(x^*,\lambda^*)
    \le
    \mathcal L_\beta(x,\lambda^*),
    \qquad
    \forall (x,\lambda)\in\mathbb R^n\times\mathbb R^m .
\end{equation}
Indeed, the first inequality follows from $Ax^*=b$, while the second follows
from the convexity of $f$ and the stationarity condition
$\nabla f(x^*)+A^\top\lambda^*=0$. Therefore, solving \eqref{prob_main} is equivalent to finding a
primal-dual solution in $\Omega$.

The goal of this paper is to design accelerated augmented Lagrangian schemes for
\eqref{prob_main} and to establish convergence of the   primal-dual sequence,
together with accelerated estimates for the augmented Lagrangian gap, feasibility
violation, objective residual, and, in the smooth case, stationarity residual.

\subsection{Related works}

Accelerated augmented Lagrangian methods for the linearly constrained convex
problem \eqref{prob_main} have attracted considerable attention in the last two
decades. A basic idea is to combine the feasibility control provided by the
augmented Lagrangian with Nesterov extrapolation. He and Yuan
\cite{HeOO} showed that the augmented Lagrangian method can be accelerated by
using Nesterov extrapolation for the dual sequence, and obtained an
$\bigO(1/k^2)$ convergence rate for the Lagrangian gap. This dual acceleration
strategy was further developed in several directions, including linearized
subproblems, Bregman regularization, and Uzawa schemes; see
\cite{HuangJSC,KangJSC,TaoMC}. These works established accelerated estimates
for the Lagrangian gap. Another line of research applies Nesterov extrapolation mainly to the
primal sequence. Sabach and Teboulle
\cite{SabachSIOPT} proposed a unifying Lagrangian-based framework and obtained
faster nonergodic estimates under a general algorithmic map. Xu \cite{XuSIOPT}
developed accelerated first-order primal-dual proximal methods for composite
linearly constrained problems. Other related accelerated primal-dual or
linearized augmented Lagrangian methods include
\cite{Tran20,Tang2020,HeADMM,BaiJCM,HeANM}. These methods show that the
feasibility violation and the objective residual can achieve $\bigO(1/k^2)$
rates. However, the acceleration is often imposed mainly on the primal sequence,
whereas the dual sequence is treated through a multiplier or correction update.
Thus the inertial acceleration of the primal and dual variables is  not always present in these algorithmic frameworks.

A complementary viewpoint of ALMs is provided by the continuous-time dynamical system
interpretation of acceleration. Continuous-time models of Nesterov acceleration are usually described
by inertial dynamics with a vanishing damping $\alpha/t$.
This viewpoint has led to many Lyapunov  analyses for accelerated
optimization methods; see
\cite{SuJMLR,WibisonoPNAS,LuoMP,XieMp,AttouchSIOPT,HeArxiv}. For
linearly constrained problems, inertial primal-dual dynamics have been studied
in \cite{ZengTAC,HeSICON,BotJDE,AttouchJOTA,HeAutomatica}. These works clarify
how primal-dual extrapolation, damping, and constraints 
interact at the continuous level. Motivated by numerical discretizations of
such inertial dynamics, several accelerated ALM-type methods have been proposed
and shown to achieve $\bigO(1/k^2)$ convergence rates. In particular, inertial discretization techniques lead to fast primal-dual methods for
linearly constrained problems, and some scaled variants further improve the
convergence rates by introducing additional scaling coefficients; see
\cite{LuoMc,LuoJota,HeAutomatica,HeNA}. 

Nevertheless, for the aforementioned Nesterov accelerated ALMs, the
available results mainly concern convergence rates of function values,
feasibility measures, Lagrangian gaps, or primal-dual gaps. The convergence of
the generated primal-dual iterate sequence is generally not addressed, except
for a few results in the strongly convex regime. This leaves a gap between
accelerated residual estimates and full last-iterate convergence for the
linearly constrained convex problem \eqref{prob_main}. The most closely related work  is the fast augmented
Lagrangian method of Bo{\c t}, Csetnek and Nguyen \cite{BotMP}. Their algorithm
is obtained from discretizations of inertial primal-dual dynamical systems
related to \cite{ZengTAC,HeSICON,BotJDE}, and is formulated in terms of the
augmented Lagrangian. It treats both primal and dual variables in an inertial
way and proves $\bigO(1/t_k^2)$ estimates for the feasibility violation and the
objective residual for three classical parameter choices: the Nesterov rule
\cite{Nesterov1983}, the Chambolle-Dossal rule \cite{ChambolleJota}, and the
Attouch-Cabot rule \cite{Attouch18siam}. Furthermore, they established
convergence of the generated iterates under the noncritical parameter condition
\[
      t_{k+1}^2-t_k^2\le \rho t_{k+1},
    \qquad
    \rho\in(0,1).
\]
However, the critical case $\rho=1$, which corresponds to the classical
Nesterov parameter setting, is not covered by their iterate convergence
result. Moreover, the rate estimates in \cite{BotMP} remain of $\bigO(1/t_k^2)$
rates and are obtained in the smooth setting.

The above discussion reveals a gap between accelerated rates and
iterate convergence of ALMs for linearly constrained convex problems. Although many
Nesterov accelerated ALMs have been shown to enjoy  $\bigO(1/k^2)$ estimates for objective
residuals, feasibility violations, or primal-dual gaps
\cite{HeADMM,HeNA,XuSIOPT,LuoMc,SabachSIOPT,TaoMC}, convergence of the 
primal-dual sequence is usually not proved without additional assumptions such
as strong convexity. To the best of our knowledge, the main available result in
this direction is due to Bo{\c t}, Csetnek and Nguyen \cite{BotMP}, where
iterate convergence was established for a fast augmented Lagrangian method in
the smooth convex setting and under a noncritical parameter regime. This leaves
several natural questions. Can one prove convergence of the whole primal-dual
sequence for convex $C^1$ objectives without assuming Lipschitz continuity of
the gradient? Does the convergence result remain valid in the critical case
$\rho=1$, which corresponds to the classical Nesterov parameter choice? In
addition, even in the noncritical smooth setting, can the usual $\bigO(1/k^2)$
bounds be sharpened to $o(1/k^2)$ rates for the feasibility violation and the
objective residual? Such little-o behavior is known for Nesterov accelerated
forward-backward algorithms in unconstrained  settings
\cite{AttouchSIOPT,HeCoap,Attouch18siam}, but has not been obtained for the
accelerated ALM framework considered here. These issues motivate the accelerated
ALM schemes developed in this paper, together with their iterate convergence and
improved convergence rates.

\subsection{Main contributions}

Motivated by the above discussion, this paper develops a unified accelerated augmented
Lagrangian method for \eqref{prob_main}. The main contributions are
summarized as follows.

\begin{enumerate}
\item[\rm(i)]
We propose a unified accelerated ALM in which both the primal and dual variables
are equipped with Nesterov extrapolation. The method is governed by a
general sequence $\{t_k\}$ satisfying
\[
    t_{k+1}^2-t_k^2\le \rho t_{k+1},
    \qquad
    \rho\in(0,1].
\]
Hence it covers both the critical Nesterov case $\rho=1$ and the
noncritical case $\rho<1$. In contrast to accelerated ALMs where the inertial mechanism is mainly imposed on either the dual
sequence or the primal sequence
\cite{HeOO,HuangJSC,XuSIOPT,SabachSIOPT,HeADMM,TaoMC,LuoMc}, the proposed method
uses a coupled primal-dual inertial structure. The framework contains an
implicit-gradient scheme for convex $C^1$ objectives and a partially explicit
scheme for convex $L$-smooth objectives.

\item[\rm(ii)]
We prove convergence of the   primal-dual sequence generated by the algorithm:
\[
    (x_k,\lambda_k)\to(x^*,\lambda^*)\in\Omega.
\]
Many existing accelerated ALM methods have been shown to enjoy $\bigO(1/t_k^2)$ rates for
objective residuals, feasibility violations or Lagrangian gaps under different choices of $\{t_k\}_{k\ge 1}$, but the convergence of the
generated primal-dual iterates remains open
\cite{XuSIOPT,SabachSIOPT,HeNA}. The fast ALM of Bo{\c t}, Csetnek and
Nguyen \cite{BotMP} is one of the few works proving iterate convergence, but
their result is obtained in the smooth setting and under a noncritical parameter
regime $\rho<1$. By contrast, our convergence analysis also covers the convex
$C^1$ implicit-gradient case and the critical case $\rho=1$. Moreover, under
the general parameter condition above, we recover the accelerated estimates
\[
    \|Ax_k-b\|
    =
    O\left(\frac1{t_k^2}\right),
    \qquad
    |f(x_k)-f(x^*)|
    =
    O\left(\frac1{t_k^2}\right).
\]
For the classical parameter choices where $t_k$ has the same order as $k$, these
bounds become the usual $\bigO(1/k^2)$ rates.

\item[\rm(iii)]
In the noncritical regime $0<\rho<1$, we improve the standard
$\bigO(1/t_k^2)$ estimates to little-o rates:
\[
    \|Ax_k-b\|
    =
    o\left(\frac1{t_k^2}\right),
    \qquad
    |f(x_k)-f(x^*)|
    =
    o\left(\frac1{t_k^2}\right).
\]
For classical noncritical choices of $\{t_k\}$, this
gives the sharper $o(1/k^2)$ rates, and improves the usual big-O rates obtained
in Nesterov accelerated ALMs 
\cite{HeOO,XuSIOPT,BotMP,SabachSIOPT,HeNA,LuoMc}. In the smooth case, we further
derive the stationarity residual estimate
\[
    \|\nabla f(x_k)+A^\top\lambda_k\|
    =
    o\left(\frac1{t_k}\right),
\]
which improves the $o(1/\sqrt{t_k})$ estimate in \cite{BotMP}.  Such little-o rates have not been previously established for accelerated augmented Lagrangian-type methods.   The same
Lyapunov argument also yields direct variants for composite differentiable
objectives and for a scaled primal-dual formulation.
\end{enumerate}

\subsection{Organization}

The rest of the paper is organized as follows. Section \ref{sec2}
introduces the accelerated augmented Lagrangian methods. Section
\ref{sec3} establishes the unified Lyapunov estimate. Section
\ref{sec4} proves the convergence rates, the convergence of the
 primal-dual sequence, and the improved little-o rates. Section \ref{sec5} reports numerical experiments  to illustrate the theoretical results.
Section \ref{sec6} concludes the paper, and the appendix collects auxiliary
results.

\section{Accelerated augmented Lagrangian methods}\label{sec2}

In this section, we introduce a unified accelerated augmented Lagrangian method. The construction is motivated by time discretizations of
an inertial primal-dual dynamical system associated with the augmented
Lagrangian formulation of \eqref{prob_main}. We first recall the continuous-time
model and its discretization motivation, and then present the unified coupled
scheme and its algorithm form.

We start from the augmented inertial primal-dual dynamical system considered in
\cite{ZengTAC,HeSICON,BotJDE,HeArxiv}:
\begin{equation}\label{dy_main}
\begin{cases}
\ddot{x}(t)+\dfrac{\alpha}{t}\dot{x}(t)
+\nabla f(x(t))
+A^\top\bigl(\lambda(t)+\theta t\dot{\lambda}(t)\bigr)
+\beta A^\top(Ax(t)-b)=0,
\\[0.8em]
\ddot{\lambda}(t)+\dfrac{\alpha}{t}\dot{\lambda}(t)
-\bigl(A(x(t)+\theta t\dot{x}(t))-b\bigr)=0.
\end{cases}
\end{equation}
Here the term $\beta A^\top(Ax(t)-b)$ is the augmented feasibility correction,
while
\[
    \lambda(t)+\theta t\dot\lambda(t),
    \qquad
    x(t)+\theta t\dot x(t)
\]
represent the primal-dual extrapolation mechanism.

Let $\{\tau_k\}_{k\ge0}$ be a uniform grid with stepsize $h>0$, and write
$(x_k,\lambda_k)\approx (x(\tau_k),\lambda(\tau_k))$. At time $\tau_{k}$, we
approximate the inertial terms by
\begin{equation*}
\begin{cases}
\ddot x(\tau_{k})
    +\dfrac{\alpha}{\tau_{k}}\dot x(\tau_k)
    \approx
    \dfrac{x_{k+1}-2x_k+x_{k-1}}{h^2}
    +
    \dfrac{\alpha}{\tau_{k}}
    \dfrac{x_k-x_{k-1}}{h},
\\[0.8em]
\ddot\lambda(\tau_{k})
    +\dfrac{\alpha}{\tau_{k}}\dot\lambda(\tau_k)
    \approx
    \dfrac{\lambda_{k+1}-2\lambda_k+\lambda_{k-1}}{h^2}
    +
    \dfrac{\alpha}{\tau_{k}}
    \dfrac{\lambda_k-\lambda_{k-1}}{h}.
\end{cases}
\end{equation*}
Define
\[
    (\widehat x_k,\widehat \lambda_k)
    :=
    (x_k,\lambda_k)
    +
    \left(1-\frac{\alpha h}{\tau_{k}}\right)
    \bigl[(x_k,\lambda_k)-(x_{k-1},\lambda_{k-1})\bigr].
\]
Then the preceding approximations can be written as
\[
    \ddot x(\tau_{k})
    +\frac{\alpha}{\tau_{k}}\dot x(\tau_k)
    \approx
    \frac{x_{k+1}-\widehat x_k}{h^2},
    \qquad
    \ddot\lambda(\tau_{k})
    +\frac{\alpha}{\tau_{k}}\dot\lambda(\tau_k)
    \approx
    \frac{\lambda_{k+1}-\widehat\lambda_k}{h^2}.
\]
For the primal-dual correction terms, we use backward differences:
\[
    \theta\tau_{k}\dot\lambda(\tau_{k})
    \approx
    \frac{\theta\tau_{k}}{h}(\lambda_{k+1}-\lambda_k),
    \qquad
    \theta\tau_{k}\dot x(\tau_{k})
    \approx
    \frac{\theta\tau_{k}}{h}(x_{k+1}-x_k).
\]
Set $\widehat\alpha_k:=\frac{\theta\tau_{k}}{h}$. A semi-implicit
discretization of \eqref{dy_main} then leads to the unified scheme
\begin{equation}\label{dy_disc}
\begin{cases}
x_{k+1}-\widehat x_k
=
-h^2\left[
g_{k+1}
+A^\top\bigl(\lambda_{k+1}
+\widehat\alpha_k(\lambda_{k+1}-\lambda_k)\bigr)
+\beta A^\top(Ax_{k+1}-b)
\right],
\\[1.2ex]
\lambda_{k+1}-\widehat\lambda_k
=
h^2\left[
Ax_{k+1}-b+\widehat\alpha_k A(x_{k+1}-x_k)
\right],
\end{cases}
\end{equation}
where
\[
    g_{k+1}
    =
    \begin{cases}
    \nabla f(x_{k+1}), & \text{implicit-gradient},\\[0.4ex]
    \nabla f(\widehat x_k), & \text{explicit-gradient}.
    \end{cases}
\]
The implicit choice is appropriate when $f$ is convex and continuously
differentiable. The explicit choice is available under the additional Lipschitz
continuity of $\nabla f$; it linearizes only the objective term, while the
augmented penalty term is still evaluated implicitly at $x_{k+1}$.

Different choices of the $\{\tau_k\}$ and of the finite-difference
approximation may lead to different Nesterov extrapolation coefficients in
the resulting accelerated schemes \cite{HeNA,BotMP,LuoMc}. Instead of fixing a
particular discretization rule in \eqref{dy_disc}, we directly consider a
general class of extrapolation coefficients described by a positive sequence
$\{t_k\}_{k\ge1}$ satisfies the following assumption.

\begin{assumption}\label{ass_tk}
The sequence $\{t_k\}_{k\ge1}$ is nondecreasing, $t_1=1$, $t_k>1$ for all
$k>2$, and $t_k\to+\infty$ as $k\to+\infty$. Moreover,
\[
      t_{k+1}^2-t_k^2\le \rho t_{k+1},
    \qquad
    \rho\in(0,1].
\]
\end{assumption}

\begin{remark}
Assumption \ref{ass_tk} covers several standard inertial parameter rules used in
accelerated first-order methods. In particular, it includes the Nesterov
 rule \cite{Nesterov1983}, the Chambolle-Dossal  rule \cite{ChambolleJota}, and
the Attouch-Cabot rule  \cite{Attouch18siam}. More general inertial rules and
the parameter choices appearing in fast augmented Lagrangian methods are also
covered; see  \cite{BotMP,HeNA,LuoMc,HeADMM,XuSIOPT}.
\end{remark}

\begin{remark}\label{re_tk}
When $\{t_k\}_{k\ge 1}$ satisfies Assumption \ref{ass_tk}, Lemma \ref{lem_scaledtk} with
$\xi_k=1$ gives
$
    \sum_{k=1}^{+\infty}1/t_k=+\infty.
$
If the growth condition holds with equality, namely
\[
    t_{k+1}^2-t_k^2=\rho t_{k+1},
\]
then $
    (t_{k+1}-t_k)(t_{k+1}+t_k)=\rho t_{k+1},
$
and hence
$
    t_{k+1}-t_k=\frac{\rho t_{k+1}}{t_{k+1}+t_k}.
$
Since $t_{k+1}\le t_{k+1}+t_k\le 2t_{k+1}$, we have
$
    \frac{\rho}{2}\le t_{k+1}-t_k\le \rho.
$
Summing this inequality from $1$ to $k-1$ gives
\[
    t_1+\frac{\rho}{2}(k-1)\le t_k\le t_1+\rho(k-1).
\]
Thus $\{t_k\}$ has the same order as $\{k\}$ as $k\to+\infty$.
\end{remark}

Let $\{t_k\}_{k\ge1}$ satisfy Assumption~\ref{ass_tk}. We define the
extrapolated points by
\[
    \bar x_k
    :=
    x_k+\frac{t_k-1}{t_{k+1}}(x_k-x_{k-1}),
    \qquad
    \bar\lambda_k
    :=
    \lambda_k+\frac{t_k-1}{t_{k+1}}(\lambda_k-\lambda_{k-1}),
\]
and choose
\[
    \alpha_k:=\frac{t_{k+1}-\eta}{\eta},
    \qquad \eta\in[\rho,1].
\]
Then
\[
    \eta\alpha_k=t_{k+1}-\eta,
\]
which is the key cancellation relation used in the Lyapunov analysis. Replacing
$\widehat x_k,\widehat\lambda_k,\widehat\alpha_k$ in \eqref{dy_disc} by
$\bar x_k,\bar\lambda_k,\alpha_k$, and allowing different primal and dual
stepsizes $\gamma>0$ and $\delta>0$, we obtain the unified coupled scheme
\begin{equation}\label{alg_diff}
\begin{cases}
x_{k+1}-\bar x_k
=
-\gamma\left[
g_{k+1}
+A^\top\bigl(\lambda_{k+1}
+\alpha_k(\lambda_{k+1}-\lambda_k)\bigr)
+\beta A^\top(Ax_{k+1}-b)
\right],
\\[1.2ex]
\lambda_{k+1}-\bar\lambda_k
=
\delta\left[
Ax_{k+1}-b+\alpha_k A(x_{k+1}-x_k)
\right],
\end{cases}
\end{equation}
where
\begin{equation}\label{eq_def_g0}
    g_{k+1}
    =
    \begin{cases}
    \nabla f(x_{k+1}), & \text{\rm Case I: } f \text{ is convex and } C^1,\\[0.4ex]
    \nabla f(\bar x_k), & \text{\rm Case II: } f \text{ is convex and } L\text{-smooth}.
    \end{cases}
\end{equation}
Here $C^1$ means that $f$ is continuously differentiable, while $L$-smooth
means that $\nabla f$ is $L$-Lipschitz continuous, namely,
\[
    \|\nabla f(x)-\nabla f(y)\|\le L\|x-y\|,
    \qquad \forall x,y\in\mathbb{R}^n.
\]
Thus Case I uses an implicit gradient evaluation at $x_{k+1}$, whereas Case II
uses an explicit gradient evaluation at the extrapolated point $\bar x_k$.

The coupled scheme \eqref{alg_diff} contains $x_{k+1}$ and $\lambda_{k+1}$
simultaneously. We now eliminate $\lambda_{k+1}$ to obtain a computable primal
subproblem followed by an explicit multiplier update.

Define
\[
    c_k:=1+\alpha_k=\frac{t_{k+1}}{\eta}.
\]
Given $\bar\lambda_k$, set
\[
    p_k:=c_k\bar\lambda_k-\alpha_k\lambda_k,
    \qquad
    r_k:=\alpha_kAx_k+b.
\]
Then the dual equation in \eqref{alg_diff} gives
\[
    \lambda_{k+1}
    =
    \bar\lambda_k+\delta(c_kAx_{k+1}-r_k).
\]
Moreover,
\begin{equation}\label{eq_dual_corr}
\begin{aligned}
\lambda_{k+1}+\alpha_k(\lambda_{k+1}-\lambda_k)
&=
c_k\lambda_{k+1}-\alpha_k\lambda_k
=
p_k+\delta c_k(c_kAx_{k+1}-r_k).
\end{aligned}
\end{equation}
Substituting \eqref{eq_dual_corr} into the primal equation of \eqref{alg_diff} gives the following
unified implementable algorithm (Algorithm \ref{al_main}).

\begin{algorithm}[!htbp]
\caption{Accelerated augmented Lagrangian method}
\label{al_main}
\begin{algorithmic}[1]
\Require Initial points $x_0=x_1\in\mathbb R^n$ and
$\lambda_0=\lambda_1\in\mathbb R^m$; parameters $\gamma>0$, $\delta>0$,
$\beta\ge0$; a sequence $\{t_k\}_{k\ge1}$ satisfying Assumption
\ref{ass_tk} with $\rho\in(0,1]$; a parameter $\eta\in[\rho,1]$.
\For{$k=1,2,\ldots$}
    \State Compute
    \[
        \alpha_k=\frac{t_{k+1}-\eta}{\eta},
        \qquad
        c_k=\frac{t_{k+1}}{\eta}.
    \]
    \State Compute the extrapolated points
    \[
        \bar x_k
        =
        x_k+\frac{t_k-1}{t_{k+1}}(x_k-x_{k-1}),
        \qquad
        \bar\lambda_k
        =
        \lambda_k+\frac{t_k-1}{t_{k+1}}(\lambda_k-\lambda_{k-1}).
    \]
    \State Set
    \[
        p_k=c_k\bar\lambda_k-\alpha_k\lambda_k,
        \qquad
        r_k=\alpha_kAx_k+b.
    \]
    \State Compute $x_{k+1}$ by one of the following  rules.

    \smallskip
    \noindent
    \textbf{Case I: convex $C^1$ objective.}
    \[
    \begin{aligned}
    x_{k+1}
    =
    \arg\min_{x\in\mathbb R^n}
    \Bigg\{
    & f(x)
    +\frac{\beta}{2}\|Ax-b\|^2
    +\frac{1}{2\gamma}\|x-\bar x_k\|^2
    +\langle p_k,Ax-b\rangle
    +\frac{\delta}{2}\|c_kAx-r_k\|^2
    \Bigg\}.
    \end{aligned}
    \]

    \smallskip
    \noindent
    \textbf{Case II: convex $L$-smooth objective.}
    Assume $\gamma\le 1/L$. 
    \[
    \begin{aligned}
    x_{k+1}
    =
    \arg\min_{x\in\mathbb R^n}
    \Bigg\{
    &
    \langle \nabla f(\bar x_k),x\rangle
    +\frac{\beta}{2}\|Ax-b\|^2
    +\frac{1}{2\gamma}\|x-\bar x_k\|^2
    +\langle p_k,Ax-b\rangle
    +\frac{\delta}{2}\|c_kAx-r_k\|^2
    \Bigg\}.
    \end{aligned}
    \]

    \State Update the multiplier by
    \[
        \lambda_{k+1}
        =
        \bar\lambda_k+\delta(c_kAx_{k+1}-r_k).
    \]
\EndFor
\end{algorithmic}
\end{algorithm}

Both primal subproblems are strongly convex because of the proximal term. In
Case I, the   objective $f$ is treated implicitly, so no Lipschitz
continuity of $\nabla f$ is required. In Case II, only the objective term is
linearized, while the augmented penalty term remains implicit. Thus the primal
step keeps an augmented Lagrangian structure and retains direct control of
feasibility. The derivation above gives the following equivalence between Algorithm
\ref{al_main} and the coupled scheme \eqref{alg_diff}.

\begin{proposition}\label{prop00}
The sequence generated by Algorithm \ref{al_main} is equivalent to the sequence
satisfying the coupled scheme \eqref{alg_diff}. More precisely, for the same
initial points and parameters, $\{(x_k,\lambda_k)\}_{k\ge0}$ is generated by
Algorithm \ref{al_main} if and only if it satisfies \eqref{alg_diff} with
$g_{k+1}$ defined by \eqref{eq_def_g0}.
\end{proposition}

Consequently, in the convergence analysis, we shall work with the sequence
generated by Algorithm \ref{al_main} and use the coupled relations
\eqref{alg_diff}.

 \section{A unified energy estimate}\label{sec3}

In this section, we derive a Lyapunov estimate that will be used throughout the
subsequent convergence and rate analysis. Let
$\{(x_k,\lambda_k)\}_{k\ge0}$ be the sequence generated by Algorithm
\ref{al_main}, and set
\[
    z_k:=(x_k,\lambda_k),
    \qquad
    \bar z_k:=(\bar x_k,\bar\lambda_k).
\]
Define the positive definite matrix
\begin{equation}\label{eq_M}
    M:=
    \begin{pmatrix}
    \gamma^{-1}I_n & 0\\
    0 & \delta^{-1}I_m
    \end{pmatrix}.
\end{equation}
For any $z=(x,\lambda)\in\mathbb R^n\times\mathbb R^m$, we use the notation
\[
    \|z\|_M^2
    :=
    \langle Mz,z\rangle
    =
    \frac1\gamma\|x\|^2+\frac1\delta\|\lambda\|^2.
\]
For a fixed primal-dual solution $z^*=(x^*,\lambda^*)\in\Omega$, define the augmented
Lagrangian gap
\begin{equation}\label{eq_Hbeta}
\begin{aligned}
    H_k^{z^*}
    &:=
    \mathcal L_\beta(x_k,\lambda^*)
    -
    \mathcal L_\beta(x^*,\lambda^*)    \\
    &=
    f(x_k)-f(x^*)
    +\langle \lambda^*,Ax_k-b\rangle
    +\frac{\beta}{2}\|Ax_k-b\|^2 .
\end{aligned}
\end{equation}
By the saddle point inequality \eqref{eq:saddle}, one has
$H_k^{z^*}\ge0$ for all $k\ge0$.

The estimates below are based on the coupled relations \eqref{alg_diff}, which
are equivalent to Algorithm \ref{al_main} by Proposition
\ref{prop00}. We first derive an identity for the inertial
part of the energy. This identity is independent of whether Case I or Case II
is used.

\begin{lemma}\label{le_ener_B}
Let $\{(x_k,\lambda_k)\}_{k\ge0}$ be the sequence generated by Algorithm
\ref{al_main}, and let $z^*=(x^*,\lambda^*)\in\Omega$ be fixed. Define
\begin{equation}\label{eq_Bdef}
    B_k^{z^*}
    :=
    \frac12\|w_k^{z^*}\|_M^2
    +
    \frac{\eta(1-\eta)}{2}\|z_k-z^*\|_M^2,
    \qquad k\ge1,
\end{equation}
where
\begin{equation}\label{eq_w_def}
    w_k^{z^*}
    :=
    \eta(z_k-z^*)+(t_k-1)(z_k-z_{k-1}).
\end{equation}
Then, for every $k\ge1$, the following identity holds:
\begin{align}\label{eq_Bdiff}
B_{k+1}^{z^*}-B_k^{z^*}
={}&
-\eta t_{k+1}
\Big\langle x_{k+1}-x^*,
g_{k+1}
+\beta A^\top(Ax_{k+1}-b)
+A^\top\lambda^*
\Big\rangle
\notag\\
&-
t_{k+1}(t_{k+1}-\eta)
\Big\langle x_{k+1}-x_k,
g_{k+1}
+\beta A^\top(Ax_{k+1}-b)
+A^\top\lambda^*
\Big\rangle
\notag\\
&-
(1-\eta)\left(t_{k+1}-\frac12\right)
\|z_{k+1}-z_k\|_M^2
-\frac{t_{k+1}^2}{2}
\|z_{k+1}-\bar z_k\|_M^2 .
\end{align}
\end{lemma}

\begin{proof}
By the definition of $\bar z_k$, we have
\[
    t_{k+1}(z_{k+1}-\bar z_k)
    =
    t_{k+1}(z_{k+1}-z_k)
    -(t_k-1)(z_k-z_{k-1}).
\]
Consequently,
\begin{align}\label{eq_w_diff}
w_{k+1}^{z^*}-w_k^{z^*}
&=
\eta(z_{k+1}-z_k)
+(t_{k+1}-1)(z_{k+1}-z_k)
-(t_k-1)(z_k-z_{k-1})
\notag\\
&=
t_{k+1}(z_{k+1}-\bar z_k)
-(1-\eta)(z_{k+1}-z_k).
\end{align}
It follows that
\begin{align*}
\frac12\|w_{k+1}^{z^*}-w_k^{z^*}\|_M^2
={}&
\frac{t_{k+1}^2}{2}\|z_{k+1}-\bar z_k\|_M^2
+\frac{(1-\eta)^2}{2}\|z_{k+1}-z_k\|_M^2
\\
&-
(1-\eta)t_{k+1}
\left\langle
M(z_{k+1}-z_k),z_{k+1}-\bar z_k
\right\rangle .
\end{align*}
Using
$
    \frac12\|a\|_M^2-\frac12\|b\|_M^2
    =
    \langle Ma,a-b\rangle
    -\frac12\|a-b\|_M^2
$
with $a=w_{k+1}^{z^*}$ and $b=w_k^{z^*}$, together with \eqref{eq_w_diff}, we
obtain
\begin{align}\label{eq_w_norm}
\frac12\|w_{k+1}^{z^*}\|_M^2
-\frac12\|w_k^{z^*}\|_M^2
={}&
\langle Mw_{k+1}^{z^*},w_{k+1}^{z^*}-w_k^{z^*}\rangle
-\frac12\|w_{k+1}^{z^*}-w_k^{z^*}\|_M^2
\notag\\
={}&
\eta t_{k+1}
\langle M(z_{k+1}-z^*),z_{k+1}-\bar z_k\rangle
\notag\\
&+
t_{k+1}(t_{k+1}-\eta)
\langle M(z_{k+1}-z_k),z_{k+1}-\bar z_k\rangle
\notag\\
&-
\eta(1-\eta)
\langle M(z_{k+1}-z^*),z_{k+1}-z_k\rangle
\notag\\
&-
(1-\eta)\left(t_{k+1}-\frac{1+\eta}{2}\right)
\|z_{k+1}-z_k\|_M^2
-\frac{t_{k+1}^2}{2}\|z_{k+1}-\bar z_k\|_M^2.
\end{align}

Applying the same identity to $z_{k+1}-z^*$ and $z_k-z^*$ gives
\begin{align}\label{eq_z_diff}
&\frac{\eta(1-\eta)}{2}\|z_{k+1}-z^*\|_M^2
-\frac{\eta(1-\eta)}{2}\|z_k-z^*\|_M^2
\notag\\
&\qquad={}
\eta(1-\eta)
\langle M(z_{k+1}-z^*),z_{k+1}-z_k\rangle
-
\frac{\eta(1-\eta)}{2}\|z_{k+1}-z_k\|_M^2.
\end{align}
Adding \eqref{eq_w_norm} and \eqref{eq_z_diff}, the mixed terms cancel, and
therefore
\begin{align}\label{eq_B_diff2}
B_{k+1}^{z^*}-B_k^{z^*}
={}&
\eta t_{k+1}
\langle M(z_{k+1}-z^*),z_{k+1}-\bar z_k\rangle
\notag\\
&+
t_{k+1}(t_{k+1}-\eta)
\langle M(z_{k+1}-z_k),z_{k+1}-\bar z_k\rangle
\notag\\
&-
(1-\eta)\left(t_{k+1}-\frac12\right)
\|z_{k+1}-z_k\|_M^2
-\frac{t_{k+1}^2}{2}\|z_{k+1}-\bar z_k\|_M^2.
\end{align}

It remains to rewrite the first two inner products by using the coupled scheme
\eqref{alg_diff}. Since $Ax^*=b$, the primal equation gives
\[
\frac1\gamma(x_{k+1}-\bar x_k)
=
-\left[
g_{k+1}
+\beta A^\top(Ax_{k+1}-b)
+A^\top\lambda^*
+A^\top\bigl(\lambda_{k+1}-\lambda^*
+\alpha_k(\lambda_{k+1}-\lambda_k)\bigr)
\right],
\]
and the dual equation gives
\[
\frac1\delta(\lambda_{k+1}-\bar\lambda_k)
=
A(x_{k+1}-x^*)+\alpha_kA(x_{k+1}-x_k).
\]
Thus, by the definition of $M$,
\begin{align}\label{eq_le1_1}
\langle M(z_{k+1}-z^*),z_{k+1}-\bar z_k\rangle
={}&
\frac{1}{\gamma}\langle x_{k+1}-x^*,x_{k+1}-\bar x_k\rangle
+
\frac{1}{\delta}\langle \lambda_{k+1}-\lambda^*,\lambda_{k+1}-\bar \lambda_k\rangle
\notag\\
={}&
-\Big\langle x_{k+1}-x^*,
g_{k+1}
+\beta A^\top(Ax_{k+1}-b)
+A^\top\lambda^*
\Big\rangle
\notag\\
&+
\alpha_k
\Big[
-\langle Ax_{k+1}-b,\lambda_{k+1}-\lambda_k\rangle
+
\langle \lambda_{k+1}-\lambda^*,A(x_{k+1}-x_k)\rangle
\Big].
\end{align}
Similarly,
\begin{align}\label{eq_le1_2}
\langle M(z_{k+1}-z_k),z_{k+1}-\bar z_k\rangle={}&
-\Big\langle x_{k+1}-x_k,
g_{k+1}
+\beta A^\top(Ax_{k+1}-b)
+A^\top\lambda^*
\Big\rangle
\notag\\
&-
\langle A(x_{k+1}-x_k),\lambda_{k+1}-\lambda^*\rangle
+
\langle \lambda_{k+1}-\lambda_k,Ax_{k+1}-b\rangle.
\end{align}
Multiplying \eqref{eq_le1_1} by $\eta t_{k+1}$ and \eqref{eq_le1_2} by
$t_{k+1}(t_{k+1}-\eta)$, the coupling terms cancel because
$\eta\alpha_k=t_{k+1}-\eta$. Hence
\begin{align*}
&\eta t_{k+1}
\langle M(z_{k+1}-z^*),z_{k+1}-\bar z_k\rangle
+
t_{k+1}(t_{k+1}-\eta)
\langle M(z_{k+1}-z_k),z_{k+1}-\bar z_k\rangle
\\
&\qquad=
-\eta t_{k+1}
\Big\langle x_{k+1}-x^*,
g_{k+1}
+\beta A^\top(Ax_{k+1}-b)
+A^\top\lambda^*
\Big\rangle
\\
&\quad\qquad-
t_{k+1}(t_{k+1}-\eta)
\Big\langle x_{k+1}-x_k,
g_{k+1}
+\beta A^\top(Ax_{k+1}-b)
+A^\top\lambda^*
\Big\rangle.
\end{align*}
Substituting this identity into \eqref{eq_B_diff2} gives \eqref{eq_Bdiff}.
\end{proof}

We now estimate the full Lyapunov sequence. Although Case I and Case II use
different gradient evaluations, the final energy inequality has the same form in
both cases.

\begin{lemma}\label{lem_energys}
Let $\{(x_k,\lambda_k)\}_{k\ge0}$ be the sequence generated by Algorithm
\ref{al_main}, and let $z^*=(x^*,\lambda^*)\in\Omega$ be fixed. Define the
energy sequence by
\begin{equation}\label{eq_E_def}
    \mathcal E_k^{z^*}
    :=
    t_k^2H_k^{z^*}+B_k^{z^*},
\end{equation}
where $H_k^{z^*}$ is defined in \eqref{eq_Hbeta} and $B_k^{z^*}$ is defined in
\eqref{eq_Bdef}. Then, for both Case I and Case II,
\[
\mathcal E_{k+1}^{z^*}-\mathcal E_k^{z^*}
\le
(\rho-\eta)t_{k+1}H_k^{z^*}
-
(1-\eta)\left(t_{k+1}-\frac12\right)
\|z_{k+1}-z_k\|_M^2.
\]
\end{lemma}

\begin{proof}
We prove the two cases separately.

\medskip
\noindent
\textbf{Case I.}
In this case, $g_{k+1}=\nabla f(x_{k+1})$. By the convexity of $f$ and
$Ax^*=b$, we obtain
\begin{align}\label{eq_le1_3}
&\Big\langle x_{k+1}-x^*,
\nabla f(x_{k+1})
+\beta A^\top(Ax_{k+1}-b)
+A^\top\lambda^*
\Big\rangle
\notag\\
&\qquad\ge
f(x_{k+1})-f(x^*)
+\langle \lambda^*,Ax_{k+1}-b\rangle
+\beta\|Ax_{k+1}-b\|^2
\notag\\
&\qquad\ge
H_{k+1}^{z^*}.
\end{align}
Similarly, again by convexity,
\begin{align}\label{eq_le1_4}
&\Big\langle x_{k+1}-x_k,
\nabla f(x_{k+1})
+\beta A^\top(Ax_{k+1}-b)
+A^\top\lambda^*
\Big\rangle
\notag\\
&\quad\ge
f(x_{k+1})-f(x_k)
+\langle \lambda^*,A(x_{k+1}-x_k)\rangle
+\beta\langle A(x_{k+1}-x_k),Ax_{k+1}-b\rangle
\notag\\
&\quad\ge H_{k+1}^{z^*}-H_k^{z^*},
\end{align}
where the last inequality follows from
\[
\langle A(x_{k+1}-x_k),Ax_{k+1}-b\rangle
=
\frac12\Big(
\|Ax_{k+1}-b\|^2-\|Ax_k-b\|^2
+\|A(x_{k+1}-x_k)\|^2
\Big).
\]
Combining Lemma \ref{le_ener_B} with \eqref{eq_le1_3} and \eqref{eq_le1_4}, we
get
\begin{align*}
\mathcal E_{k+1}^{z^*}-\mathcal E_k^{z^*}
&\le
t_{k+1}^2H_{k+1}^{z^*}-t_k^2H_k^{z^*}
+B_{k+1}^{z^*}-B_k^{z^*}
\\
&\le
\left(t_{k+1}^2-t_k^2-\eta t_{k+1}\right)H_k^{z^*}
\\
&\quad
-
(1-\eta)\left(t_{k+1}-\frac12\right)
\|z_{k+1}-z_k\|_M^2
-\frac{t_{k+1}^2}{2}
\|z_{k+1}-\bar z_k\|_M^2 .
\end{align*}
Using Assumption \ref{ass_tk} and dropping the last nonpositive term gives the
desired estimate.

\medskip
\noindent
\textbf{Case II.}
In this case, $g_{k+1}=\nabla f(\bar x_k)$. Since $f$ is $L$-smooth, Lemma
\ref{le_eqsmf} gives
\[
    \langle x_{k+1}-x^*,\nabla f(\bar x_k)\rangle
    \ge
    f(x_{k+1})-f(x^*)
    -\frac{L}{2}\|x_{k+1}-\bar x_k\|^2
\]
and
\[
    \langle x_{k+1}-x_k,\nabla f(\bar x_k)\rangle
    \ge
    f(x_{k+1})-f(x_k)
    -\frac{L}{2}\|x_{k+1}-\bar x_k\|^2.
\]
Therefore, by the same argument as in Case I,
\[
\Big\langle x_{k+1}-x^*,
\nabla f(\bar x_k)
+\beta A^\top(Ax_{k+1}-b)
+A^\top\lambda^*
\Big\rangle
\ge
H_{k+1}^{z^*}
-\frac{L}{2}\|x_{k+1}-\bar x_k\|^2
\]
and
\[
\Big\langle x_{k+1}-x_k,
\nabla f(\bar x_k)
+\beta A^\top(Ax_{k+1}-b)
+A^\top\lambda^*
\Big\rangle
\ge
H_{k+1}^{z^*}-H_k^{z^*}
-\frac{L}{2}\|x_{k+1}-\bar x_k\|^2.
\]
Combining these two inequalities with Lemma \ref{le_ener_B} yields
\begin{align}\label{eq_diffE2}
\mathcal E_{k+1}^{z^*}-\mathcal E_k^{z^*}
&\le
\left(t_{k+1}^2-t_k^2-\eta t_{k+1}\right)H_k^{z^*}
+\frac{Lt_{k+1}^2}{2}\|x_{k+1}-\bar x_k\|^2
\notag\\
&\quad
-
(1-\eta)\left(t_{k+1}-\frac12\right)
\|z_{k+1}-z_k\|_M^2
-\frac{t_{k+1}^2}{2}
\|z_{k+1}-\bar z_k\|_M^2 .
\end{align}
By \eqref{eq_M},
\[
\|z_{k+1}-\bar z_k\|_M^2
=
\frac{1}{\gamma}\|x_{k+1}-\bar x_k\|^2
+
\frac{1}{\delta}\|\lambda_{k+1}-\bar \lambda_k\|^2
\ge
\frac{1}{\gamma}\|x_{k+1}-\bar x_k\|^2.
\]
Since $\gamma\le 1/L$, the last term in \eqref{eq_diffE2} dominates the
positive term involving $L\|x_{k+1}-\bar x_k\|^2/2$. Using Assumption
\ref{ass_tk}, we obtain the desired estimate.
\end{proof}

\begin{remark}
Lemma \ref{lem_energys} is the basic Lyapunov estimate for the subsequent
analysis. It implies the monotonicity of $\{\mathcal E_k^{z^*}\}_{k\ge1}$ when
$\eta\ge\rho$, and hence gives the boundedness of the generated sequence.
Moreover, when $0<\rho<\eta<1$, the resulting weighted summability estimates
will be used to prove convergence of the   primal-dual sequence and the
improved $o(1/t_k^2)$ rates.
\end{remark}

 \section{Convergence analysis}\label{sec4}

In this section, we establish the convergence properties of Algorithm
\ref{al_main}. We first derive the basic convergence estimates from the unified
energy inequality in Lemma \ref{lem_energys}. We then prove convergence of the
  primal-dual sequence. Finally, in the noncritical parameter regime, we
improve the standard $\bigO(1/t_k^2)$ estimates to $o(1/t_k^2)$ rates and derive a
stationarity residual estimate in the smooth case.

\subsection{Convergence rates}

We first establish the basic convergence estimates for Algorithm \ref{al_main}.
They include weighted summability, boundedness, and accelerated rates for the
augmented Lagrangian gap, feasibility violation, and objective residual. These
estimates will be used in the subsequent sequence convergence and little-o
analysis.

\begin{theorem}\label{thm_Orates}
Let $\{(x_k,\lambda_k)\}_{k\ge0}$ be generated by Algorithm \ref{al_main}, and
fix $(x^*,\lambda^*)\in\Omega$. Then the following conclusions hold.
\begin{enumerate}
\item[\rm(i)] Weighted summability:
\[
    (\eta-\rho)\sum_{k=1}^{+\infty}
    t_{k+1}(\mathcal L_\beta(x_k,\lambda^*)
    -\mathcal L_\beta(x^*,\lambda^*))<+\infty,
\]
\[
     (1-\eta) \sum_{k=1}^{+\infty}
    t_{k+1}\|(x_{k+1},\lambda_{k+1})-(x_{k},\lambda_{k})\|_M^2<+\infty.
\]

\item[\rm(ii)]
The augmented Lagrangian gap satisfies
\[
    \mathcal L_\beta(x_k,\lambda^*)
    -\mathcal L_\beta(x^*,\lambda^*)
    \le
    \frac{\mathcal E_1^{z^*}}{t_k^2},
    \qquad k\ge1,
\]
where
\[
    \mathcal E_1^{z^*}
    =
    \mathcal L_\beta(x_1,\lambda^*)
    -
    \mathcal L_\beta(x^*,\lambda^*)
    +\frac{\eta}{2\gamma}\|x_1-x^*\|^2
    +\frac{\eta}{2\delta}\|\lambda_1-\lambda^*\|^2.
\]

\item[\rm(iii)]
The sequence $\{(x_k,\lambda_k)\}_{k\ge 0}$ is bounded, and
\[
    \|(x_k,\lambda_k)-(x_{k-1},\lambda_{k-1})\|_M
    \le
    \frac{2\sqrt{2\mathcal E_1^{z^*}}}{t_k-1},
    \qquad k>1.
\]

\item[\rm(iv)]
There exists a constant $C>0$ such that the feasibility violation and the
objective residual satisfy
\[
    \|Ax_k-b\|
    \le
    \frac{C}{t_k^2},
    \qquad
    |f(x_k)-f(x^*)|
    \le
    \frac{
    \mathcal E_1^{z^*}
    +\|\lambda^*\|C
    +\beta C^2/2}{t_k^2},
    \qquad k\ge1.
\]
\end{enumerate}
\end{theorem}

\begin{proof}
Set $z_k=(x_k,\lambda_k)$ and $z^*=(x^*,\lambda^*)$. By Lemma
\ref{lem_energys}, we have
\begin{equation}\label{eq_E_diff2}
\mathcal E_{k+1}^{z^*}-\mathcal E_k^{z^*}
\le
(\rho-\eta)t_{k+1}H_k^{z^*}
-
(1-\eta)\left(t_{k+1}-\frac12\right)
\|z_{k+1}-z_k\|_M^2 .
\end{equation}
Since $H_k^{z^*}\ge0$, $B_k^{z^*}\ge0$, and $0<\rho\le\eta\le1$, it follows that
$\mathcal E_k^{z^*}\ge0$ and
\begin{equation}\label{eq_diffE_bound}
    \mathcal E_{k+1}^{z^*}\le \mathcal E_k^{z^*},
    \qquad k\ge1.
\end{equation}
Thus $\{\mathcal E_k^{z^*}\}_{k\ge1}$ is nonincreasing and bounded from below.
Summing \eqref{eq_E_diff2} from $k=1$ to $N$ and letting $N\to+\infty$ gives
\({\rm (i)}\).

From \eqref{eq_diffE_bound} and the definition of $\mathcal E_k^{z^*}$ in
\eqref{eq_E_def}, we have
\[
    t_k^2H_k^{z^*}
    \le
    \mathcal E_k^{z^*}
    \le
    \mathcal E_1^{z^*}.
\]
Together with \eqref{eq_Hbeta}, this proves \({\rm (ii)}\).

Next, by \eqref{eq_diffE_bound} and the definition of $\mathcal E_k^{z^*}$ and $\eta<\le 1$, we
obtain
\begin{equation}\label{eq_wbound}
    \|w_k^{z^*}\|_M
    \le \sqrt{2B_k^{z^*}}
    \le
    \sqrt{2\mathcal E_1^{z^*}},
    \qquad k\ge1.
\end{equation}
Applying Lemma \ref{le_wbound} with $a_k=t_k-1$ and
$W=\sqrt{2\mathcal E_1^{z^*}}$, and using \eqref{eq_w_def} and
\eqref{eq_wbound}, we obtain
\[
    \|z_k-z^*\|_M
    \le
    \frac{\sqrt{2\mathcal E_1^{z^*}}}{\eta},
    \qquad k\ge1,
\]
and
\[
    (t_k-1)\|z_k-z_{k-1}\|_M
    \le
    2\sqrt{2\mathcal E_1^{z^*}},
    \qquad k\ge1.
\]
This proves \({\rm (iii)}\).

It remains to estimate the feasibility violation and the objective residual.
Since $\alpha_k=(t_{k+1}-\eta)/\eta$, the dual equation in \eqref{alg_diff}
gives
\[
\lambda_{k+1}-\bar\lambda_k
=
\frac{\delta}{\eta}
\left[
t_{k+1}(Ax_{k+1}-b)
-
(t_{k+1}-\eta)(Ax_k-b)
\right].
\]
Multiplying both sides by $t_{k+1}$ gives
\begin{equation}\label{eq_dual_diff}
t_{k+1}(\lambda_{k+1}-\bar\lambda_k)
=
\frac{\delta}{\eta}
\left[
r_{k+1}
-
\frac{t_{k+1}(t_{k+1}-\eta)}{t_k^2}r_k
\right],
\end{equation}
where
\[
    r_k:=t_k^2(Ax_k-b).
\]
Define
\begin{equation}\label{eq_ak0}
    a_k
    :=
    1-\frac{t_{k+1}(t_{k+1}-\eta)}{t_k^2}
    =
    \frac{\eta t_{k+1}-(t_{k+1}^2-t_k^2)}{t_k^2}.
\end{equation}
Using the definition of $\bar\lambda_k$, \eqref{eq_dual_diff} can be rewritten
as
\[
    r_{k+1}-r_k+a_kr_k
    =
    \frac{\eta}{\delta}
    \left[
    t_{k+1}(\lambda_{k+1}-\lambda_k)
    -(t_k-1)(\lambda_k-\lambda_{k-1})
    \right].
\]
Summing this identity from $j=1$ to $k$ yields
\begin{align}\label{eq_rdiff}
r_{k+1}-r_1
+\sum_{j=1}^{k}a_jr_j
=
\frac{\eta}{\delta}
\left[
\lambda_{k+1}-\lambda_1
+
(t_{k+1}-1)(\lambda_{k+1}-\lambda_k)
\right].
\end{align}
Indeed,
\[
\begin{aligned}
&\sum_{j=1}^k
\left[
t_{j+1}(\lambda_{j+1}-\lambda_j)
-(t_j-1)(\lambda_j-\lambda_{j-1})
\right]
\\
&\qquad =
\sum_{j=1}^k
(\lambda_{j+1}-\lambda_j)
+
\sum_{j=1}^k
\left[
(t_{j+1}-1)(\lambda_{j+1}-\lambda_j)
-(t_j-1)(\lambda_j-\lambda_{j-1})
\right]
\\
&\qquad =
\lambda_{k+1}-\lambda_1+(t_{k+1}-1)(\lambda_{k+1}-\lambda_k),
\end{aligned}
\]
where the initialization $\lambda_1=\lambda_0$ has been used.

By \({\rm (iii)}\), the sequence $\{\lambda_k\}$ is bounded and
$\{(t_{k+1}-1)\|\lambda_{k+1}-\lambda_k\|\}$ is bounded. Hence the right-hand
side of \eqref{eq_rdiff} is bounded. Therefore, there exists a constant
$C_0>0$ such that
\begin{equation}\label{eq_rsum}
    \left\|
    r_{k+1}
    +
    \sum_{j=1}^{k}a_jr_j
    \right\|
    \le C_0,
    \qquad k\ge1.
\end{equation}
Moreover, by \eqref{eq_ak0} and Assumption \ref{ass_tk}, together with
$\rho\le\eta$, we have $a_k\ge0$. Also
\[
    1-a_k=
    \frac{t_{k+1}(t_{k+1}-\eta)}{t_k^2}>0.
\]
Thus $a_k\in[0,1)$. Applying Lemma \ref{le_heauto} to \eqref{eq_rsum}, we get
\[
   \sup_{k\ge1}\|t_k^2(Ax_k-b)\|
   =
   \sup_{k\ge1}\|r_k\|<+\infty.
\]
Therefore, there exists a constant $C>0$ such that
\[
    \|Ax_k-b\|
    \le
    \frac{C}{t_k^2},
    \qquad k\ge1.
\]
Finally, by the definition of $H_k^{z^*}$ in \eqref{eq_Hbeta} and \({\rm (ii)}\),
we obtain
\[
\begin{aligned}
|f(x_k)-f(x^*)|
&\le
H_k^{z^*}
+
\|\lambda^*\|\,\|Ax_k-b\|
+
\frac{\beta}{2}\|Ax_k-b\|^2
\\
&\le
\frac{\mathcal E_1^{z^*}}{t_k^2}
+
\frac{\|\lambda^*\|C}{t_k^2}
+
\frac{\beta C^2}{2t_k^4}.
\end{aligned}
\]
Since $t_k\ge1$, we have $t_k^{-4}\le t_k^{-2}$. This proves \({\rm (iv)}\).
The proof is complete.
\end{proof}

\begin{remark}\label{rem:comparison-basic-rates}
Theorem \ref{thm_Orates} gives $\bigO(1/t_k^2)$ estimates under the general
parameter condition in Assumption \ref{ass_tk}. Compared with the fast ALM of
Bo{\c t}, Csetnek and Nguyen \cite{BotMP} and the primal-dual accelerated ALM of
He, Hu and Fang \cite{HeNA}, where primal-dual Nesterov extrapolation is
also considered but the corresponding rates are proved for  Nesterov rule  \cite{Nesterov1983} with $t_{k+1}=\frac{1+\sqrt{1+4t_k^2}}{2}$, Chamboll-Dossal  rule \cite{ChambolleJota} with $t_k=\frac{k+\alpha-2}{\alpha-1}$ and  Attouch-Cabot  rule  \cite{Attouch18siam} with $t_k=\frac{k-1}{\alpha-1}$, Algorithm \ref{al_main} allows a broader class of
Nesterov sequences $\{t_k\}$ while keeping inertia on both the primal and
dual variables.
\end{remark}

 \subsection{Convergence of the primal-dual sequence}

We next prove convergence of the  primal-dual sequence. By Theorem
\ref{thm_Orates}, the generated sequence is bounded and hence has cluster
points. The proof consists of two parts: we first show that every cluster point
belongs to the KKT set $\Omega$, and then prove that the set of cluster points is
a singleton.

\begin{lemma}\label{lem_saddle}
Let $\{(x_k,\lambda_k)\}_{k\ge0}$ be generated by Algorithm \ref{al_main}. Then
every cluster point of $\{(x_k,\lambda_k)\}_{k\ge 0}$ belongs to $\Omega$.
\end{lemma}

\begin{proof}
Fix an arbitrary point $z^*:=(x^*,\lambda^*)\in\Omega$. We prove the result in
three steps.

\medskip
\noindent
\textbf{Step 1.} We first prove that
\begin{equation}\label{eq_nf_con}
    \nabla f(x_k)\to \nabla f(x^*).
\end{equation}
Since $t_k\to+\infty$, by Theorem \ref{thm_Orates}, the sequence $\{x_k\}$ is bounded. Moreover,
\begin{equation}\label{eq_facon}
    \|Ax_k-b\|\to0,
    \qquad
    f(x_k)\to f(x^*).
\end{equation}
Let $\bar x$ be an arbitrary cluster point of $\{x_k\}$. Then there exists a
subsequence $\{x_{k_j}\}$ such that $x_{k_j}\to\bar x$. Passing to the limit in
\eqref{eq_facon}, and using the continuity of $A$ and $f$, we get
\[
    A\bar x=b,
    \qquad
    f(\bar x)=f(x^*).
\]
Hence $\bar x$ is a primal solution. Since $f$ is convex and continuously
differentiable and $\Omega\neq\emptyset$, there exists some multiplier
$\bar\lambda$ such that $(\bar x,\bar\lambda)\in\Omega$. Therefore, Proposition
\ref{prop_App1} gives
\[
    \nabla f(\bar x)=\nabla f(x^*).
\]

We now show that the   sequence $\{\nabla f(x_k)\}$ converges to
$\nabla f(x^*)$. Suppose, by contradiction, that \eqref{eq_nf_con} does not
hold. Then there exist $\varepsilon_0>0$ and a subsequence $\{x_{k_j}\}$ such
that
\begin{equation}\label{eq_diff_f}
    \|\nabla f(x_{k_j})-\nabla f(x^*)\|\ge \varepsilon_0,
    \qquad \forall j\ge1.
\end{equation}
Since $\{x_k\}$ is bounded, the subsequence $\{x_{k_j}\}$ has a further
subsequence, still denoted by $\{x_{k_j}\}$, such that $x_{k_j}\to\bar x$ for
some cluster point $\bar x$ of $\{x_k\}$. As shown above,
$\nabla f(\bar x)=\nabla f(x^*)$. Thus, by the continuity of $\nabla f$,
\[
    \nabla f(x_{k_j})\to \nabla f(\bar x)=\nabla f(x^*),
\]
which contradicts \eqref{eq_diff_f}. Hence \eqref{eq_nf_con} holds.

\medskip
\noindent
\textbf{Step 2.} We next prove that
\begin{equation}\label{eq_alamba}
    A^\top\lambda_k\to A^\top\lambda^*.
\end{equation}
From the primal update in the coupled form \eqref{alg_diff}, we have
\begin{equation}\label{eq_saddle_1}
x_{k+1}-\bar x_k
=
-\gamma\left[
g_{k+1}
+\beta A^\top(Ax_{k+1}-b)
+A^\top\bigl(\lambda_{k+1}
+\alpha_k(\lambda_{k+1}-\lambda_k)\bigr)
\right],
\end{equation}
where $g_{k+1}$ is defined in \eqref{eq_def_g0}. Since
$\alpha_k=t_{k+1}/\eta-1$, we have
\[
    \lambda_{k+1}+\alpha_k(\lambda_{k+1}-\lambda_k)
    =
   \frac{t_{k+1}}{\eta}\lambda_{k+1}-\left(\frac{t_{k+1}}{\eta}-1\right)\lambda_k.
\]
Using $\nabla f(x^*)+A^\top\lambda^*=0$, we rewrite
\eqref{eq_saddle_1} as
\begin{equation}\label{eq_Alam}
A^\top\lambda_{k+1}-A^\top\lambda^*
=
\left(1-\frac{\eta}{t_{k+1}}\right)
\left(A^\top\lambda_k-A^\top\lambda^*\right)
+
\frac{\eta}{t_{k+1}}q_k,
\end{equation}
where
\begin{equation}\label{eq_qk_s}
    q_k
    :=
    -\left(g_{k+1}-\nabla f(x^*)\right)
    -\beta A^\top(Ax_{k+1}-b)
    -\frac1\gamma(x_{k+1}-\bar x_k).
\end{equation}

We claim that $q_k\to0$. By Theorem \ref{thm_Orates} and $t_k\to+\infty$, we
have $\|Ax_k-b\|\to0$ and $\|x_k-x_{k-1}\|\to0$. Moreover,
\begin{equation}\label{eq_diffxxbar}
\|x_{k+1}-\bar x_k\|
\le
\|x_{k+1}-x_k\|
+
\frac{t_k-1}{t_{k+1}}\|x_k-x_{k-1}\|
\to0.
\end{equation}
In Case I, $g_{k+1}=\nabla f(x_{k+1})$, and Step 1 gives
$g_{k+1}\to \nabla f(x^*)$. In Case II, $g_{k+1}=\nabla f(\bar x_k)$. Since
$\nabla f$ is Lipschitz continuous,
\[
\begin{aligned}
\|\nabla f(\bar x_k)-\nabla f(x^*)\|
&\le
\|\nabla f(\bar x_k)-\nabla f(x_{k+1})\|
+
\|\nabla f(x_{k+1})-\nabla f(x^*)\|
\\
&\le
L\|\bar x_k-x_{k+1}\|
+
\|\nabla f(x_{k+1})-\nabla f(x^*)\|.
\end{aligned}
\]
Together with \eqref{eq_diffxxbar} and Step 1, this gives
$g_{k+1}\to\nabla f(x^*)$ also in Case II. Hence $q_k\to0$ in both cases.

Set $\omega_k:=\|A^\top\lambda_k-A^\top\lambda^*\|$ and
$\sigma_k:=\eta/t_{k+1}$. Taking norms in \eqref{eq_Alam} gives
\begin{equation}\label{eq_akk}
    \omega_{k+1}\le (1-\sigma_k)\omega_k+\sigma_k\|q_k\|.
\end{equation}
Since $0<\eta\le1$ and $t_{k+1}\ge1$, one has $0<\sigma_k\le1$. Moreover, by
Remark \ref{re_tk},
\[
    \sum_{k=1}^{+\infty}\sigma_k
    =
    \sum_{k=1}^{+\infty}\frac{\eta}{t_{k+1}}
    =
    +\infty.
\]
Since $q_k\to0$, Lemma \ref{le_bertsekas} applied to \eqref{eq_akk} gives
$\omega_k\to0$, which proves \eqref{eq_alamba}.

\medskip
\noindent
\textbf{Step 3.} Let $(\bar x,\bar\lambda)$ be an arbitrary cluster point of
$\{(x_k,\lambda_k)\}_{k\ge0}$. Then there exists a subsequence $\{k_j\}$ such
that
\[
   (x_{k_j},\lambda_{k_j})\to(\bar x,\bar \lambda).
\]
By Theorem \ref{thm_Orates}, $\|Ax_k-b\|\to0$, and hence $A\bar x=b$. Moreover,
by Step 1, Step 2, and $\nabla f(x^*)+A^\top\lambda^*=0$,
\[
\begin{aligned}
\nabla f(x_k)+A^\top\lambda_k
&=
\bigl(\nabla f(x_k)-\nabla f(x^*)\bigr)
+
\bigl(A^\top\lambda_k-A^\top\lambda^*\bigr)
\to0.
\end{aligned}
\]
Passing to the limit along $\{k_j\}$ and using the continuity of $\nabla f$, we
obtain
\[
    \nabla f(\bar x)+A^\top\bar\lambda=0.
\]
Together with $A\bar x=b$, this gives $(\bar x,\bar\lambda)\in\Omega$.

Since the cluster point was arbitrary, every cluster point of
$\{(x_k,\lambda_k)\}_{k\ge0}$ belongs to $\Omega$. The proof is complete.
\end{proof}

Lemma \ref{lem_saddle} shows that every cluster point of the generated sequence
is a primal-dual solution. We now prove that the   primal-dual sequence converges to
a point in $\Omega$.

\begin{theorem}\label{thm_cov}
Let $\{(x_k,\lambda_k)\}_{k\ge0}$ be generated by Algorithm \ref{al_main}.
Then there exists $(x^*,\lambda^*)\in\Omega$ such that
\[
    (x_k,\lambda_k)\to (x^*,\lambda^*),\qquad \text{as } k\to+\infty.
\]
\end{theorem}

\begin{proof}
Set $z_k:=(x_k,\lambda_k)$. By Theorem \ref{thm_Orates}, $\{z_k\}_{k\ge0}$ is
bounded, and hence it has at least one cluster point. By Lemma
\ref{lem_saddle}, every cluster point of $\{z_k\}_{k\ge0}$ belongs to $\Omega$.
It remains to prove that the cluster point is unique.

Assume, by contradiction, that $\{z_k\}_{k\ge0}$ has two distinct cluster points
$z^a=(x^a,\lambda^a)$ and $z^b=(x^b,\lambda^b)$. Then
$z^a,z^b\in\Omega$. Hence there exist subsequences $\{z_{s_j}\}$ and
$\{z_{t_j}\}$ such that
\[
    z_{s_j}\to z^a,
    \qquad
    z_{t_j}\to z^b.
\]

For any $z^*=(x^*,\lambda^*)\in\Omega$, recall the energy
$\mathcal E_k^{z^*}$. By Lemma \ref{lem_energys},
$\{\mathcal E_k^{z^*}\}_{k\ge1}$ is nonincreasing and bounded from below.
Therefore, for every $z^*\in\Omega$, the limit
$\lim_{k\to+\infty}\mathcal E_k^{z^*}$ exists. In particular,
\[
    \lim_{k\to+\infty}\mathcal E_k^{z^a},
    \qquad
    \lim_{k\to+\infty}\mathcal E_k^{z^b}
\]
exist. Consequently,
\begin{equation}\label{eq_Phi}
    \Psi_k:=
    \frac1\eta\left(\mathcal E_k^{z^a}-\mathcal E_k^{z^b}\right)
\end{equation}
has a finite limit.

We now simplify the difference of the two energies. Since $z^a,z^b\in\Omega$,
we have $Ax^a=Ax^b=b$ and $f(x^a)=f(x^b)$. Moreover, Proposition
\ref{prop_App1} gives $A^\top\lambda^a=A^\top\lambda^b$. Hence, by
\eqref{eq_Hbeta},
\begin{equation}\label{eq_diffHk}
H_k^{z^a}-H_k^{z^b}
=
\langle \lambda^a-\lambda^b,Ax_k-b\rangle
=
\langle A^\top(\lambda^a-\lambda^b),x_k-x^a\rangle
=0.
\end{equation}
Therefore
\begin{eqnarray}\label{eq_diffEab}
	 \mathcal E_k^{z^{a}}-\mathcal E_k^{z^{b}}
    &=&
    B_k^{z^{a}}-B_k^{z^{b}}\\
  &=&  \frac{\eta}{2}\left(
    \|z_k-z^{a}\|_M^2-\|z_k-z^{b}\|_M^2  \right)-\eta(t_k-1)\ip{M(z^{a}-z^{b})}{z_k-z_{k-1}}\notag.
\end{eqnarray}
Define
\begin{equation}\label{eq_hkk}
    h_k:=
    \frac12\left(
    \|z_k-z^a\|_M^2-\|z_k-z^b\|_M^2
    \right).
\end{equation}
Then we have 
\[
    h_k-h_{k-1}
    =
    \left\langle M(z_k-z_{k-1}),z^b-z^a\right\rangle
\]
Combining this  with \eqref{eq_Phi} - \eqref{eq_diffEab}, we get
\[ 
    \Psi_k=  \frac{1}{\eta}(B_k^{z^a}-B_k^{z^b})
    =
    h_k+(t_k-1)(h_k-h_{k-1}),
\]
or equivalently,
\begin{equation}\label{eq_have}
    h_k
    =
    \left(1-\frac1{t_k}\right)h_{k-1}
    +
    \frac1{t_k}\Psi_k,\qquad k\ge1.
\end{equation}

Since $\Psi_k$ has a finite limit, say $\Psi_k\to\Psi_\infty$, \eqref{eq_have}
implies
\[
|h_k-\Psi_\infty|
\le
\left(1-\frac1{t_k}\right)|h_{k-1}-\Psi_\infty|
+
\frac1{t_k}|\Psi_k-\Psi_\infty|.
\]
By Assumption \ref{ass_tk} and  Remark \ref{re_tk}, we have $0<1/t_k\le 1$  and $\sum_{k=1}^{+\infty}1/t_k=+\infty$. Applying Lemma
\ref{le_bertsekas} gives
\[
    h_k\to\Psi_\infty.
\]
Thus $\{h_k\}$ has a unique limit.

On the other hand, along the subsequences $z_{s_j}\to z^a$ and
$z_{t_j}\to z^b$, \eqref{eq_hkk} gives
\[
    h_{s_j}\to
    -\frac12\|z^a-z^b\|_M^2,
    \qquad
    h_{t_j}\to
    \frac12\|z^a-z^b\|_M^2.
\]
Since $h_k$ has a unique limit, these two limits must be equal. Hence
\[
    \|z^a-z^b\|_M^2=0.
\]
Since $M$ is positive definite, $z^a=z^b$, which contradicts the assumption that
the two cluster points are distinct. Therefore, $\{z_k\}_{k\ge0}$ has a unique
cluster point.

Since $\{z_k\}_{k\ge0}$ is bounded and has a unique cluster point, the  
sequence converges. Denote its limit by $z^*=(x^*,\lambda^*)$. By Lemma
\ref{lem_saddle}, $z^*\in\Omega$. Therefore,
\[
    (x_k,\lambda_k)\to(x^*,\lambda^*),
    \qquad
    (x^*,\lambda^*)\in\Omega.
\]
\end{proof}

\begin{remark}
The convergence of the iterates for accelerated augmented Lagrangian-type methods in the merely convex setting has remained elusive in a long time.
Even in the unconstrained case,  the question of the convergence of the iterates for the Nesterov accelerated gradient method  has remained open until it was recently  resolved affirmatively  by Jang and Ryu \cite{Jang25} and Bo{\c t},  Fadili and Nguyen \cite{Bot25arx}. The first result on the convergence of the iterates for  accelerated augmented Lagrangian-type methods is due to  Bo{\c t}, Csetnek and Nguyen \cite{BotMP} who  proved the convergence of the  prima-dual sequence generated by a fast  augmented Lagrangian method  in the smooth and noncritical setting $\rho<1$. Compared with \cite{BotMP}, Theorem \ref{thm_cov}  establishes the convergence of the  primal-dual sequence of Algorithm \ref{al_main} and it  covers the critical Nesterov case $\rho=1$ and
the implicit-gradient case with a convex $C^1$ objective. When $A=0$ and $b=0$, Algorithm \ref{al_main} reduces to  the Nesterov accelerated gradient method for unconstrained convex optimization problems, and Theorem \ref{thm_cov} recovers the results on the convergence of the iterates  reported in \cite{Attouch18Mp,Bot25arx,Jang25}.
\end{remark}

 \subsection{Improved convergence rates in the noncritical case}

We now study the improved asymptotic behavior in the noncritical regime
$0<\rho<\eta<1$. The convergence of the   sequence proved above allows us
to improve the basic $\bigO(1/t_k^2)$ estimates to little-o rates.

\begin{theorem}\label{thm_o_rates}
Let $\{(x_k,\lambda_k)\}_{k\ge0}$ be generated by Algorithm \ref{al_main} with
$\rho\in(0,1)$ and $\eta\in(\rho,1)$. Let
$z^*=(x^*,\lambda^*)\in\Omega$ be the limit point of
$\{(x_k,\lambda_k)\}_{k\ge0}$. Then, as $k\to+\infty$,
\[
    \mathcal L_\beta(x_k,\lambda^*)
    -
    \mathcal L_\beta(x^*,\lambda^*)
    =
    o\left(\frac1{t_k^2}\right),
    \qquad
    \|(x_k,\lambda_k)-(x_{k-1},\lambda_{k-1})\|
    =
    o\left(\frac1{t_k}\right),
\]
and
\[
    \|Ax_k-b\|
    =
    o\left(\frac1{t_k^2}\right),
    \qquad
    |f(x_k)-f(x^*)|
    =
    o\left(\frac1{t_k^2}\right).
\]
\end{theorem}
\begin{proof}
By Lemma \ref{lem_energys}, the limit
$\lim_{k\to+\infty}\mathcal E_k^{z^*}$ exists. Since $z_k=(x_k,\lambda_k)\to z^*=(x^*,\lambda^*)$,  from the definition of $\mathcal E_k^{z^*}$ we have 
\begin{eqnarray}\label{eq_convE}
	\lim_{k\to+\infty}\mathcal E_k^{z^*} = \lim_{k\to+\infty} \left[t_k^2H_k^{z^*}+\frac{(t_k-1)^2}{2}\|z_k-z_{k-1}\|_M^2+\eta(t_k-1)\ip{M(z_k-z^*)}{z_k-z_{k-1}}\right]
\end{eqnarray}
exists.  By Theorem \ref{thm_Orates} (iii), $\sup_{k\ge 1}(t_k-1)\|z_k-z_{k-1}\|<+\infty$,  which together with $z_k\to z^*$  implies 
\begin{eqnarray}\label{eq_diffzz}
	\lim_{k\to+\infty}(t_k-1)|\ip{M(z_k-z^*)}{z_k-z_{k-1}}|\leq  \sup_{k\ge 1}(t_k-1)\|z_k-z_{k-1}\|\cdot\lim_{k\to+\infty}\|M(z_k-z^*)\| = 0.
\end{eqnarray}

Since $\rho\in (0,1)$, $\eta\in(\rho,1)$ and $t_{k+1}\ge t_k\ge 1$,  Theorem \ref{thm_Orates} (i) gives the weighted summability
\begin{equation}\label{eq_o_sum}
    \sum_{k=1}^{+\infty}t_{k}(\mathcal L_\beta(x_k,\lambda^*)
    -
    \mathcal L_\beta(x^*,\lambda^*))<+\infty,
    \qquad
    \sum_{k=1}^{+\infty}t_{k}\|z_k-z_{k-1}\|_M^2<+\infty.
\end{equation}
It also implies $\lim_{k\to+\infty}(1-2t_k)\|z_k-z_{k-1}\|_M^2=0$. Then it follows from \eqref{eq_diffzz} and \eqref{eq_convE} that 
\begin{eqnarray}\label{eq_convE2}
	  \lim_{k\to+\infty} \left[t_k^2(\mathcal L_\beta(x_k,\lambda^*) -
    \mathcal L_\beta(x^*,\lambda^*)) +\frac{t_k^2}{2}\|z_k-z_{k-1}\|_M^2\right] = \ell,
\end{eqnarray}
with $\ell\ge 0$. Denote 
\[G_k = \mathcal L_\beta(x_k,\lambda^*) -
    \mathcal L_\beta(x^*,\lambda^*) +\frac{1}{2}\|z_k-z_{k-1}\|_M^2.\]
Then, from \eqref{eq_o_sum} and \eqref{eq_convE2}, we have 
\[ \sum_{k=1}^{+\infty}t_{k} G_k <+\infty\qquad\text{and}\qquad   \lim_{k\to+\infty}t^2_kG_k=\ell\ge 0.  \] 
Now, we prove that $\ell=0$. If not, $\ell>0$ then there exists $k_0\ge 0$ such that for all  $k\ge k_0$, $t_k^2G_k\ge \frac{\ell}{2}$. This implies
\[ \sum_{k=k_0}^{+\infty}t_kG_k \ge \frac{\ell}{2}\sum_{k=k_0}^{+\infty}\frac{1}{t_k}=+\infty. \]
where the last inequality follows from Remark \ref{re_tk}. This contradicts
$\sum_{k=1}^{+\infty}t_kG_k<+\infty,
    $. Hence $\ell=0$, from \eqref{eq_convE2} and the positive definiteness of  $M$ we have 
\begin{equation}\label{eq_oL}
	 \mathcal L_\beta(x_k,\lambda^*)
    -
    \mathcal L_\beta(x^*,\lambda^*)
    =
    o\left(\frac1{t_k^2}\right), \quad  \|(x_k,\lambda_k)-(x_{k-1},\lambda_{k-1})\|=o\left(\frac1{t_k}\right).
\end{equation}

Next, we prove the $ o\left(\frac1{t_k^2}\right)$ convergence rate of  the objective residual and the feasibility violation. For simplify, set 
\begin{equation}\label{eq_ainf}
	  r_k:=t_k^2(Ax_k-b).
\end{equation}
From \eqref{eq_dual_diff}--\eqref{eq_rdiff}, we have
\begin{align}\label{eq_newG}
r_{k+1}
+\sum_{j=1}^{k}a_jr_j
=
\frac{\eta}{\delta}
\left[
\lambda_{k+1}-\lambda_1
+
(t_{k+1}-1)(\lambda_{k+1}-\lambda_k)
\right]
+r_1,
\end{align}
where
\[
    a_k
    =
    \frac{\eta t_{k+1}-(t_{k+1}^2-t_k^2)}{t_k^2}.
\]
Since $t_{k+1}^2-t_k^2\le\rho t_{k+1}$ and $\rho<\eta$, we have
$
    a_k
    \ge
    \frac{(\eta-\rho)t_{k+1}}{t_k^2}
    \ge
    \frac{\eta-\rho}{t_k}
    >0.
$
Moreover, $a_k<1$. Hence
\[
    \sum_{k=1}^{+\infty}a_k=+\infty.
\]
From \eqref{eq_oL}, we have $t_{k+1}\|\lambda_{k+1}-\lambda_k\|\to 0$. Together with $\lambda_k\to \lambda^*$ and \eqref{eq_newG} implies
\[\lim_{k\to+\infty}\left[r_{k+1}
+\sum_{j=1}^{k}a_jr_j\right] =\lambda^*-\lambda_1+Ax_1-b\in\mathbb R^m.\]
Then, it follows from  \eqref{eq_ainf} and  Lemma \ref{lem_F_res} that
\[
    \lim_{k\to+\infty}t_k^2\|Ax_k-b\|= \lim_{k\to+\infty}\|r_k\|=0.
\]
That is,
\[
    \|Ax_k-b\|=o\left(\frac1{t_k^2}\right).
\]
This together with \eqref{eq_oL} implies
\[
\begin{aligned}
|f(x_k)-f(x^*)|
&\le
 \mathcal L_\beta(x_k,\lambda^*)
    -
    \mathcal L_\beta(x^*,\lambda^*)
+\|\lambda^*\|\,\|Ax_k-b\|
+\frac{\beta}{2}\|Ax_k-b\|^2
 = o\left(\frac1{t_k^2}\right).
\end{aligned}
\]
\end{proof}

\begin{remark}\label{rem:o-rates-comparison}
The $\bigO(1/t_k^2)$ estimates of the feasibility violation and  the objective residual for accelerated ALM and Lagrangian-based methods have been obtained  in the literature. See e.g. \cite{HeNA,BotMP}.  We  improve in Theorem \ref{thm_o_rates} the standard $\bigO(1/t_k^2)$ rates to $o(1/t_k^2)$ rates for Algorithm \ref{al_main}. When
\[
    t_{k+1}^2-t_k^2=\rho t_{k+1},
\]
by Remark \ref{re_tk}  $t_k$ has the same order as $k$, and hence the
above estimates become $o(1/k^2)$. To the best of our knowledge, this is the
first $o(1/k^2)$ rate result of the feasibility violation and the  objective residual
for a Nesterov accelerated augmented Lagrangian-type method in the merely convex setting.
\end{remark}

We finally derive the stationarity residual estimate in the smooth case. This
uses the additional gradient information available in Case II together with the
little-o estimates obtained in Theorem \ref{thm_o_rates}.

\begin{theorem}\label{thm_saddleo}
Let $\{(x_k,\lambda_k)\}_{k\ge0}$ be generated by Algorithm \ref{al_main} in
Case II with 
$\rho\in(0,1)$ and $\eta\in(\rho,1)$. Let
$z^*=(x^*,\lambda^*)\in\Omega$ be the limit point of
$\{(x_k,\lambda_k)\}_{k\ge0}$. Then
\[
    \|\nabla f(x_k)+A^\top\lambda_k\|
    =
    o\left(\frac1{t_k}\right),
    \qquad
    \sum_{k=1}^{+\infty}
    t_{k+1}\|\nabla f(x_k)+A^\top\lambda_k\|^2<+\infty.
\]
\end{theorem}

\begin{proof}
We first estimate the gradient difference. Since $f$ is convex and $L$-smooth,
\[
    \frac1{2L}\|\nabla f(x_k)-\nabla f(x^*)\|^2
    \le
    f(x_k)-f(x^*)
    -
    \langle \nabla f(x^*),x_k-x^*\rangle .
\]
Using $\nabla f(x^*)+A^\top\lambda^*=0$ and $Ax^*=b$, we get
\[
\begin{aligned}
& f(x_k)-f(x^*)-\langle \nabla f(x^*),x_k-x^*\rangle
=f(x_k)-f(x^*)+\langle \lambda^*,Ax_k-b\rangle       \\
&\qquad= (\mathcal L_\beta(x_k,\lambda^*) -\mathcal L_\beta(x^*,\lambda^*))-\frac{\beta}{2}\|Ax_k-b\|^2\\
&\qquad\le\mathcal L_\beta(x_k,\lambda^*) -\mathcal L_\beta(x^*,\lambda^*).
\end{aligned}
\]
Then, by Theorem \ref{thm_Orates}(i) and Theorem \ref{thm_o_rates}, we obtain
\begin{equation}\label{eq_diff_nf}
    \|\nabla f(x_k)-\nabla f(x^*)\|
    =
    o\left(\frac1{t_k}\right),
    \qquad
    \sum_{k=1}^{+\infty}
    t_{k+1}\|\nabla f(x_k)-\nabla f(x^*)\|^2<+\infty .
\end{equation}

It remains to estimate $A^\top\lambda_k-A^\top\lambda^*$. Set
\[
    y_k:=t_k(A^\top\lambda_k-A^\top\lambda^*).
\]
It follows from \eqref{eq_Alam} and \eqref{eq_qk_s} that
\begin{equation}\label{eq_y_new}
    y_{k+1}
    =
    \frac{t_{k+1}-\eta}{t_k}y_k
    +
    \eta q_k,
\end{equation}
where, in Case II,
\[
    q_k
    :=
    -\left(\nabla f(\bar x_k)-\nabla f(x^*)\right)
    -\beta A^\top(Ax_{k+1}-b)
    -\frac1\gamma(x_{k+1}-\bar x_k).
\]
Since
\[
    \|x_{k+1}-\bar x_k\|
    \le
    \|x_{k+1}-x_k\|
    +
    \frac{t_k-1}{t_{k+1}}\|x_k-x_{k-1}\|,
\]
and
\[
\|\nabla f(\bar x_k)-\nabla f(x^*)\|
\le
L\|x_{k+1}-\bar x_k\|
+
\|\nabla f(x_{k+1})-\nabla f(x^*)\|,
\]
it follows from  \eqref{eq_diff_nf} and Theorem \ref{thm_o_rates}  that
\begin{equation}\label{eq_tq}
    t_k\|q_k\|\to0,
    \qquad
    \sum_{k=1}^{+\infty} t_k\|q_k\|^2<+\infty.
\end{equation}

We first prove the little-o estimate. Define
\[
    \sigma_k
    :=
    1-\frac{t_{k+1}-\eta}{t_k}
    =
    \frac{\eta-(t_{k+1}-t_k)}{t_k}.
\]
Since
\[
    t_{k+1}-t_k
    =
    \frac{t_{k+1}^2-t_k^2}{t_{k+1}+t_k}
    \le
    \rho\frac{t_{k+1}}{t_{k+1}+t_k}
    <\rho<\eta,
\]
we have $0<\sigma_k\le1$. Moreover,
\[
    \sigma_k\ge \frac{\eta-\rho}{t_k},
\]
and hence $\sum_{k=1}^{+\infty}\sigma_k=+\infty$. From \eqref{eq_y_new},
\[
    \|y_{k+1}\|
    \le
    (1-\sigma_k)\|y_k\|+\eta\sigma_k\frac{\|q_k\|}{\sigma_k}.
\]
By \eqref{eq_tq},
\[
    \frac{\|q_k\|}{\sigma_k}
    \le
    \frac{t_k\|q_k\|}{\eta-\rho}
    \to0.
\]
Applying Lemma \ref{le_bertsekas}, we obtain $\|y_k\|\to0$. Hence
\[
    \|A^\top\lambda_k-A^\top\lambda^*\|
    =
    o\left(\frac1{t_k}\right).
\]
Together with \eqref{eq_diff_nf} and the KKT condition
$\nabla f(x^*)+A^\top\lambda^*=0$, this gives
\[
\begin{aligned}
\|\nabla f(x_k)+A^\top\lambda_k\|
&\le
\|\nabla f(x_k)-\nabla f(x^*)\|
+
\|A^\top\lambda_k-A^\top\lambda^*\|
\\
&=
o\left(\frac1{t_k}\right).
\end{aligned}
\]

It remains to prove the weighted square summability. By Young's inequality and
\eqref{eq_y_new}, for any $\varepsilon_k>0$,
\[
\begin{aligned}
\|y_{k+1}\|^2
&=
\left\|
\frac{t_{k+1}-\eta}{t_k}y_k+\eta q_k
\right\|^2
\\
&\le
\left(\frac{t_{k+1}-\eta}{t_k}\right)^2\|y_k\|^2
+
\varepsilon_k\|y_k\|^2
+
\left(
\eta^2+
\frac{\eta^2(t_{k+1}-\eta)^2}{\varepsilon_k t_k^2}
\right)\|q_k\|^2 .
\end{aligned}
\]
Moreover,
\[
\begin{aligned}
t_k^2-(t_{k+1}-\eta)^2
&=
t_k^2-t_{k+1}^2+2\eta t_{k+1}-\eta^2
\\
&\ge
(2\eta-\rho)t_{k+1}-\eta^2
\\
&\ge
(\eta(2-\eta)-\rho)t_{k+1}.
\end{aligned}
\]
Set $c_0:=\eta(2-\eta)-\rho>0$ and choose
$\varepsilon_k:=c_0t_{k+1}/(2t_k^2)$. Then there exists a constant $C>0$ such
that
\[
    \|y_{k+1}\|^2
    \le
    \|y_k\|^2
    -
    \frac{c_0t_{k+1}}{2t_k^2}\|y_k\|^2
    +
    Ct_k\|q_k\|^2.
\]
Summing this inequality and using \eqref{eq_tq}, we get
\[
    \sum_{k=1}^{+\infty}
    \frac{t_{k+1}}{t_k^2}\|y_k\|^2
    <+\infty.
\]
Since $y_k=t_k(A^\top\lambda_k-A^\top\lambda^*)$, it follows that
\[
    \sum_{k=1}^{+\infty}
    t_{k+1}\|A^\top\lambda_k-A^\top\lambda^*\|^2
    <+\infty.
\]
Finally, using $\nabla f(x^*)+A^\top\lambda^*=0$,
\[
\begin{aligned}
\|\nabla f(x_k)+A^\top\lambda_k\|^2
&\le
2\|\nabla f(x_k)-\nabla f(x^*)\|^2
+
2\|A^\top\lambda_k-A^\top\lambda^*\|^2 .
\end{aligned}
\]
Multiplying by $t_{k+1}$ and summing over $k$, and using \eqref{eq_diff_nf}, we
obtain
\[
    \sum_{k=1}^{+\infty}
    t_{k+1}\|\nabla f(x_k)+A^\top\lambda_k\|^2<+\infty.
\]
The proof is complete.
\end{proof}

\begin{remark}
There are only few results on stationarity residual bounds  for accelerated ALMs  in the literature. Bo{\c t},  Csetnek, and Nguyen \cite{BotMP} derived the following  estimate for their fast ALM
\[
    \|\nabla f(x_k)+A^\top\lambda_k\|
    =
    o\left(\frac1{\sqrt{t_k}}\right).
\]
As a comparison, we establish in Theorem \ref{thm_saddleo} a sharper rate $o(1/t_k)$ for the stationarity residual under the same setting.
\end{remark}

 \subsection{Further discussions}

We now discuss several variants that can be obtained from the same discrete
energy mechanism.

\medskip
\noindent
{\bf Composite differentiable objectives and nonsmooth convex terms.}
The preceding arguments also apply, with only minor notational changes, to the
composite differentiable objective
\[
    f=f_1+f_2,
\]
where $f_1$ is convex and $C^1$, and $f_2$ is convex with an $L$-Lipschitz
continuous gradient. In this case, the gradient term in \eqref{alg_diff} is
replaced by the mixed implicit-explicit choice
\[
    g_{k+1}:=\nabla f_1(x_{k+1})+\nabla f_2(\bar x_k).
\]
The implicit part $\nabla f_1(x_{k+1})$ is treated by convexity at
$x_{k+1}$, while the explicit part $\nabla f_2(\bar x_k)$ is controlled by the
smooth inequality in Lemma \ref{le_eqsmf}. Hence the proof of Lemma
\ref{lem_energys} remains valid with the same Lyapunov sequence
\[
    \mathcal E_k^{z^*}=t_k^2H_k^{z^*}+B_k^{z^*}.
\]
Consequently, the convergence estimates, sequence convergence, and improved
rates in Theorems \ref{thm_Orates}, \ref{thm_cov}, and \ref{thm_o_rates} extend
to this mixed differentiable composite setting. The corresponding stationarity
condition becomes
\[
    \nabla f_1(x^*)+\nabla f_2(x^*)+A^\top\lambda^*=0.
\]

If $f_1$ is only closed, proper, and convex, the basic energy estimate in Lemma
\ref{lem_energys} can still be obtained by replacing
$\nabla f_1(x_{k+1})$ with some $g_{k+1}\in\partial f_1(x_{k+1})$ and using the
subgradient inequality. Thus the standard accelerated estimates in Theorem
\ref{thm_Orates}, for example
\[
    \|Ax_k-b\|=O\left(\frac1{t_k^2}\right),
    \qquad
    |f(x_k)-f(x^*)|=O\left(\frac1{t_k^2}\right),
\]
are still available. However, the proof of sequence convergence and the
improved $o(1/t_k^2)$ rates in Theorems \ref{thm_cov} and \ref{thm_o_rates}
uses Proposition \ref{prop_App1}, which relies on differentiability of the
objective and gives the constancy of $\nabla f(x^*)$ over the KKT set
$\Omega$. For a general nonsmooth term $f_1$, this property is no longer
available, since different solutions may admit different subgradients.
Therefore, without additional regularity assumptions, the present argument
yields only the basic $\bigO(1/t_k^2)$ rates in the nonsmooth composite case.

\medskip
\noindent
{\bf A scaled primal-dual variant.}
When $f$ is $C^1$ and the implicit choice $g_{k+1}=\nabla f(x_{k+1})$ is used,
the same energy argument also applies to a scaled primal-dual scheme. More
precisely, suppose that the primal and dual equations in \eqref{alg_diff} are
multiplied by the same positive scaling factor $\xi_{k+1}$. The coupled
relations become
\begin{equation}\label{eq_al_scal}
\begin{cases}
x_{k+1}-\bar x_k
=
-\gamma\xi_{k+1}\left[
\nabla f(x_{k+1})
+A^\top\bigl(\lambda_{k+1}
+\alpha_k(\lambda_{k+1}-\lambda_k)\bigr)
+\beta A^\top(Ax_{k+1}-b)
\right],
\\[1.2ex]
\lambda_{k+1}-\bar\lambda_k
=
\delta\xi_{k+1}\left[
Ax_{k+1}-b+\alpha_k A(x_{k+1}-x_k)
\right].
\end{cases}
\end{equation}
The use of the same scaling factor in the two equations is essential, because it
preserves the cancellation of the primal-dual coupling terms in the energy
estimate, and the scaling coefficient plays a crucial role in establishing the fast convergence properties \cite{AttouchSiamScal,AttouchJEMS,HeAutomatica}. Assume that $\xi_{k+1}\ge\xi_0>0$. The corresponding Lyapunov sequence
is
\[
    \mathcal E_k^{z^*}:=t_k^2\xi_k H_k^{z^*}+B_k^{z^*}.
\]
The only change in the energy proof is that the coefficient
$t_{k+1}^2-t_k^2$ is replaced by the scaled difference
\[
    t_{k+1}^2\xi_{k+1}-t_k^2\xi_k.
\]
We therefore impose the following scaled growth condition.

\begin{assumption}\label{ass_tks}
Let $\{t_k\}_{k\ge1}$ and $\{\xi_k\}_{k\ge1}$ be positive sequences such that
$t_1^2\xi_1=1$, $t_k^2\xi_k>1$ for all $k>2$, and $t_k^2\xi_k$ is
nondecreasing with $t_k^2\xi_k\to+\infty$. Assume that
\[
    t_{k+1}^2\xi_{k+1}-t_k^2\xi_k
    \le
    \rho t_{k+1}\xi_{k+1},
    \qquad
    0<\rho\le1.
\]
\end{assumption}

For $\eta\in[\rho,1]$, the same calculation gives
\[
    \mathcal E_{k+1}^{z^*}-\mathcal E_k^{z^*}
    \le
    (\rho-\eta)t_{k+1}\xi_{k+1}H_k^{z^*}
    -
    (1-\eta)\left(t_{k+1}-\frac12\right)
    \|z_{k+1}-z_k\|_M^2 .
\]
Consequently, the arguments used in Theorem \ref{thm_Orates} yield
\[
    \sum_{k=1}^{+\infty}t_{k+1}\xi_{k+1}H_k^{z^*}<+\infty,
    \qquad
    \sum_{k=1}^{+\infty}t_{k+1}\|z_{k+1}-z_k\|_M^2<+\infty,
\]
and
\[
    \|Ax_k-b\|=O\left(\frac1{t_k^2\xi_k}\right),
    \qquad
    |f(x_k)-f(x^*)|=O\left(\frac1{t_k^2\xi_k}\right).
\]
Moreover, Lemma \ref{lem_scaledtk} gives
$
    \sum_{k=1}^{+\infty}1/t_k=+\infty.
$
Thus the convergence of the iterative sequence satisfying \eqref{eq_al_scal}
can be proved by the same argument as in Theorem \ref{thm_cov}.

In the noncritical case $0<\rho<\eta<1$, the same energy-limit argument as in
Theorem \ref{thm_o_rates}, applied to
$\mathcal E_k^{z^*}=t_k^2\xi_k H_k^{z^*}+B_k^{z^*}$, yields
\[
    \mathcal L_\beta(x_k,\lambda^*)
    -
    \mathcal L_\beta(x^*,\lambda^*)
    =
    o\left(\frac1{t_k^2\xi_k}\right),
    \qquad
    \|(x_k,\lambda_k)-(x_{k-1},\lambda_{k-1})\|
    =
    o\left(\frac1{t_k}\right).
\]
The feasibility estimate is obtained by applying the same residual averaging
argument to the scaled residual
\[
    r_k:=t_k^2\xi_k(Ax_k-b).
\]
The corresponding coefficient becomes
\[
   a_k
    =
    1-
    \frac{\xi_{k+1}t_{k+1}(t_{k+1}-\eta)}
         {\xi_kt_k^2}.
\]
Using Assumption \ref{ass_tks}, we obtain
\[
    a_k
    =
    \frac{
    \eta t_{k+1}\xi_{k+1}
    -
    (t_{k+1}^2\xi_{k+1}-t_k^2\xi_k)}
    {t_k^2\xi_k}
    \ge
    \frac{(\eta-\rho)t_{k+1}\xi_{k+1}}{t_k^2\xi_k}.
\]
Hence the scaled residual averaging lemma applies and yields
\[
    t_k^2\xi_k(Ax_k-b)\to0.
\]
Consequently,
\[
    \|Ax_k-b\|
    =
    o\left(\frac1{t_k^2\xi_k}\right),
    \qquad
    |f(x_k)-f(x^*)|
    =
    o\left(\frac1{t_k^2\xi_k}\right).
\]
The unscaled algorithm studied in the previous sections corresponds to
$\xi_k\equiv1$.

 \section{Numerical experiments}\label{sec5}
 
 In this section, we present numerical experiments on a distributed logistic
regression problem to illustrate the convergence behavior of the proposed
algorithm.

We consider the following distributed optimization problem:
\begin{equation}\label{ques_X1}
\begin{split}
\min_{x\in\mathbb{R}^{pm}} \quad
& f(x)
=
\sum_{i=1}^{p}
\left[
\ln(1+e^{-c_i^\top x_i})
+\frac{\rho_i}{2}\|x_i\|^2
\right],
\\
\text{\rm s.t.}\quad
& x_i=x_j,\qquad \forall (i,j)\in\mathcal E .
\end{split}
\end{equation}
Here $\mathcal G(\mathcal V,\mathcal E,W)$ is an undirected communication
network with $p$ nodes, $x=[x_1^\top,x_2^\top,\ldots,x_p^\top]^\top\in
\mathbb R^{pm}$, $x_i\in\mathbb R^m$, $c_i\in\mathbb R^m$, and $\rho_i>0$ for
$i=1,2,\ldots,p$. In the experiments, the agents are connected by a ring network.
Then the consensus constraint $x_1=\cdots=x_p$ can be written as
\[
    L_{p\circ m}x=\mathbf 0_{pm},
\]
where $L_{p\circ m}=(I_p-W)\otimes I_m$. Hence \eqref{ques_X1} is a special case
of \eqref{prob_main} with $A=L_{p\circ m}$ and $b=\mathbf 0_{pm}$.

The objective function in \eqref{ques_X1} is convex and has a Lipschitz
continuous gradient. In particular, one may take
\[
    L=\max_{1\le i\le p}
    \left\{
    \frac{\|c_i\|^2}{4}+\rho_i
    \right\}.
\]
In our tests, each vector $c_i$ is generated from the uniform distribution on
$[0,1]^m$. We set $\rho_i=0.5$, $m=200$, and $p=10$.

We compare the proposed accelerated augmented Lagrangian method in Algorithm
\ref{al_main}, denoted by AALM, with the fast augmented Lagrangian
method in \cite[Algorithm 1]{BotMP}, denoted by FALM, and the accelerated
linearized primal-dual method in \cite[Algorithm 2]{HeNA}, denoted by ALPDM.
All three methods use inertial extrapolation for both the primal and dual
variables. For AALM and FALM, we choose the Chambolle-Dossal type parameter rule \cite{ChambolleJota}:
\[
    t_k=\frac{k+\alpha-2}{\alpha-1},
    \]
with $\alpha=10$, which is also the parameter choice used in ALPDM. The numerical results are reported in Figure \ref{fig1}. The proposed AALM shows
stable decay of the objective residual, feasibility violation, and stationarity
residual, which is consistent with the convergence properties proved in the
previous sections. Compared with FALM and ALPDM, AALM gives slightly better
performance in most of the tested quantities. Although the numerical difference
is moderate, the theoretical analysis developed in this paper
provides additional convergence information for AALM, including convergence of
the  primal-dual sequence and improved little-o rates in the noncritical
regime.

\begin{figure}[!htbp]
    \centering
    \begin{tabular}{ccc}
        \includegraphics[width=0.31\textwidth]{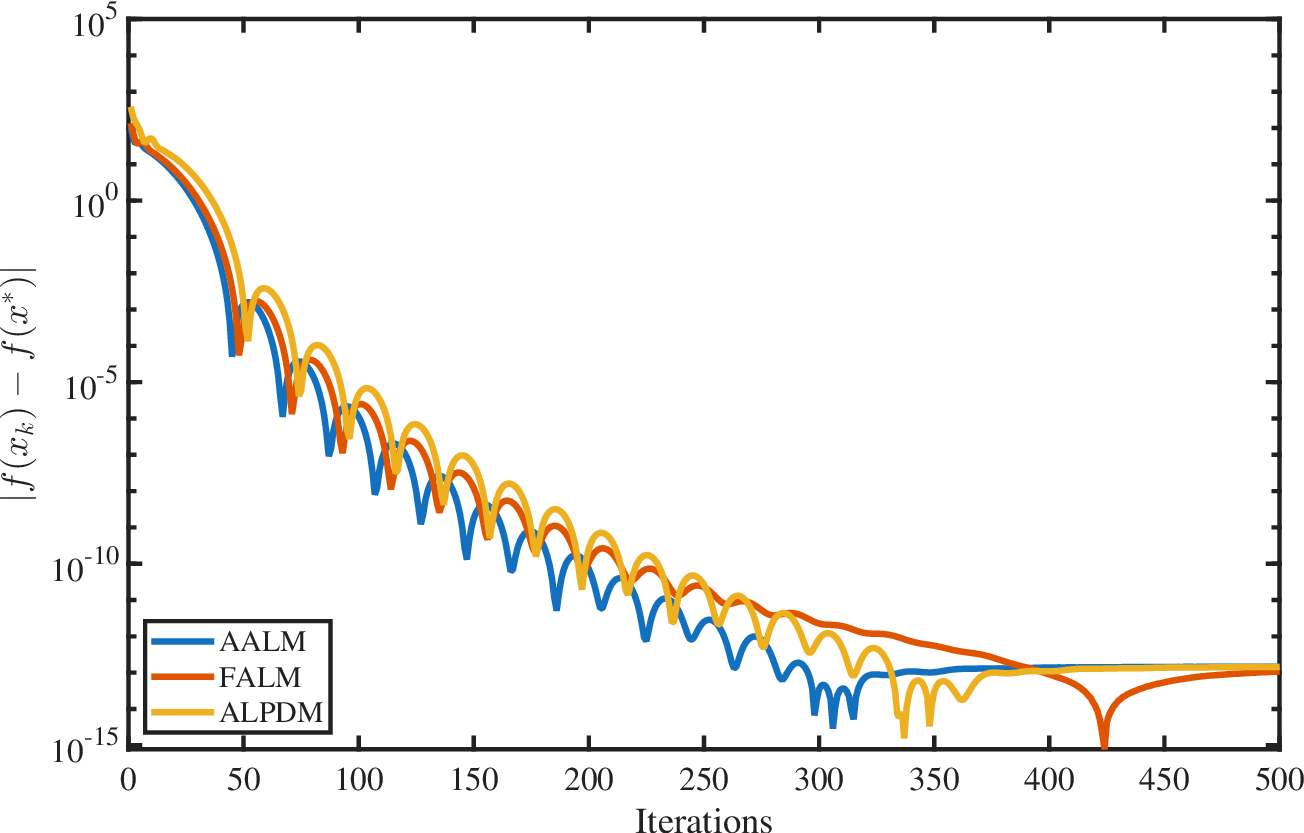} &
        \includegraphics[width=0.31\textwidth]{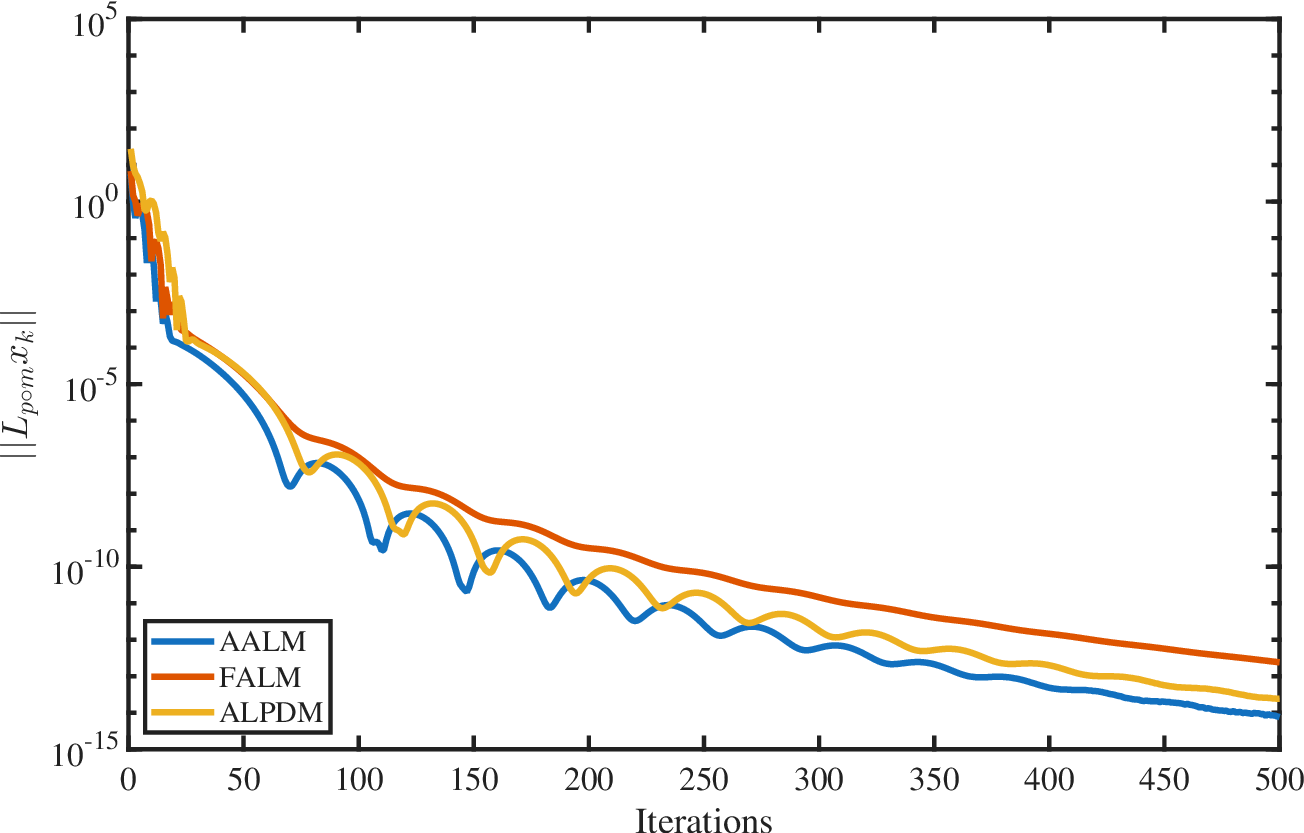} &
        \includegraphics[width=0.31\textwidth]{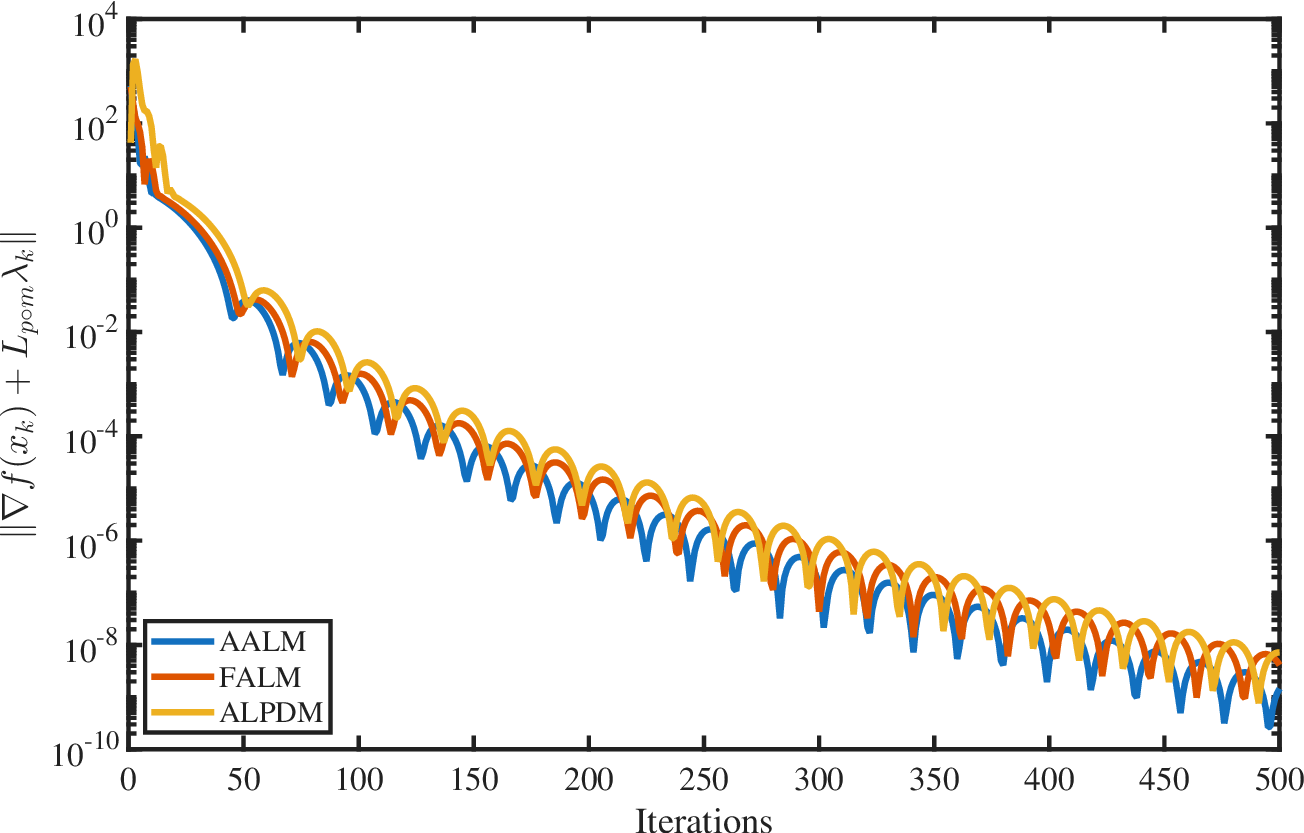} \\[0.8em]
        \includegraphics[width=0.31\textwidth]{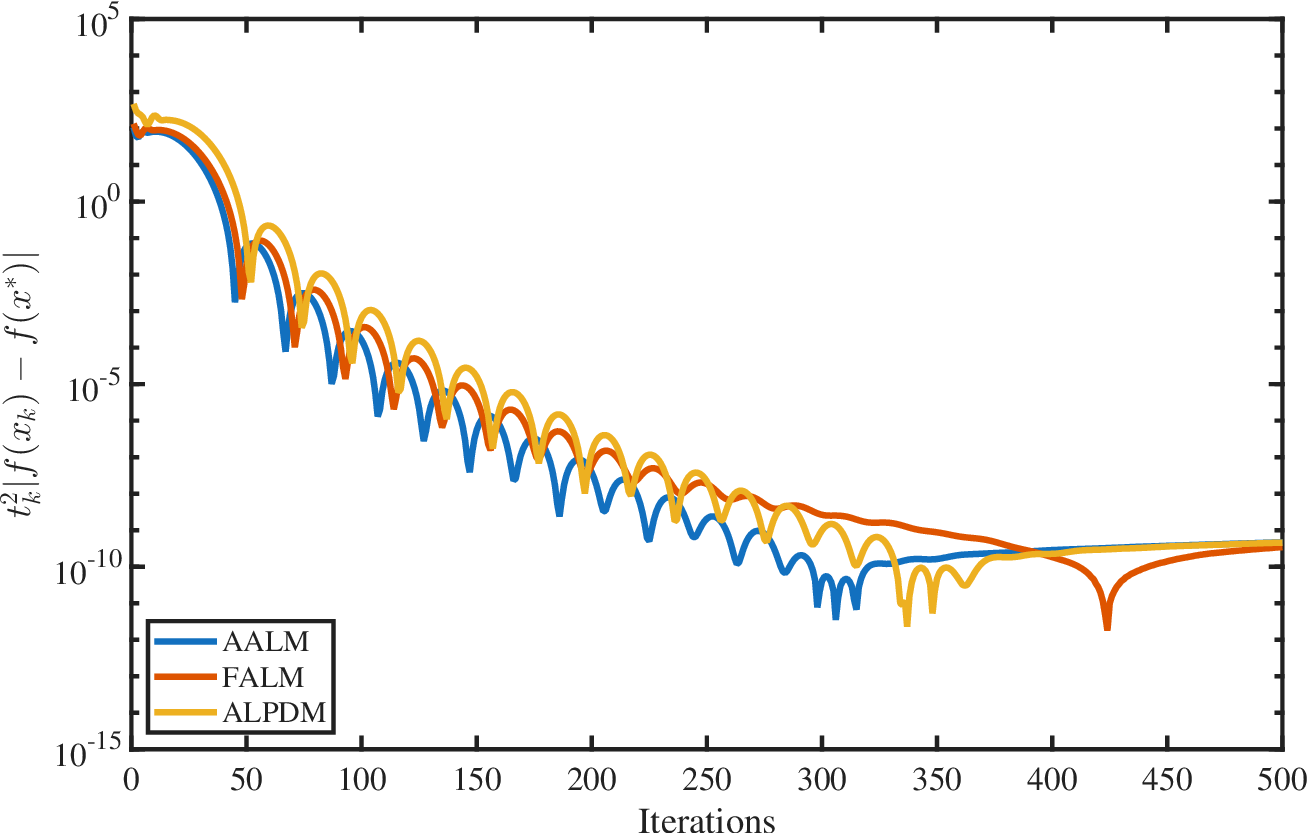} &
        \includegraphics[width=0.31\textwidth]{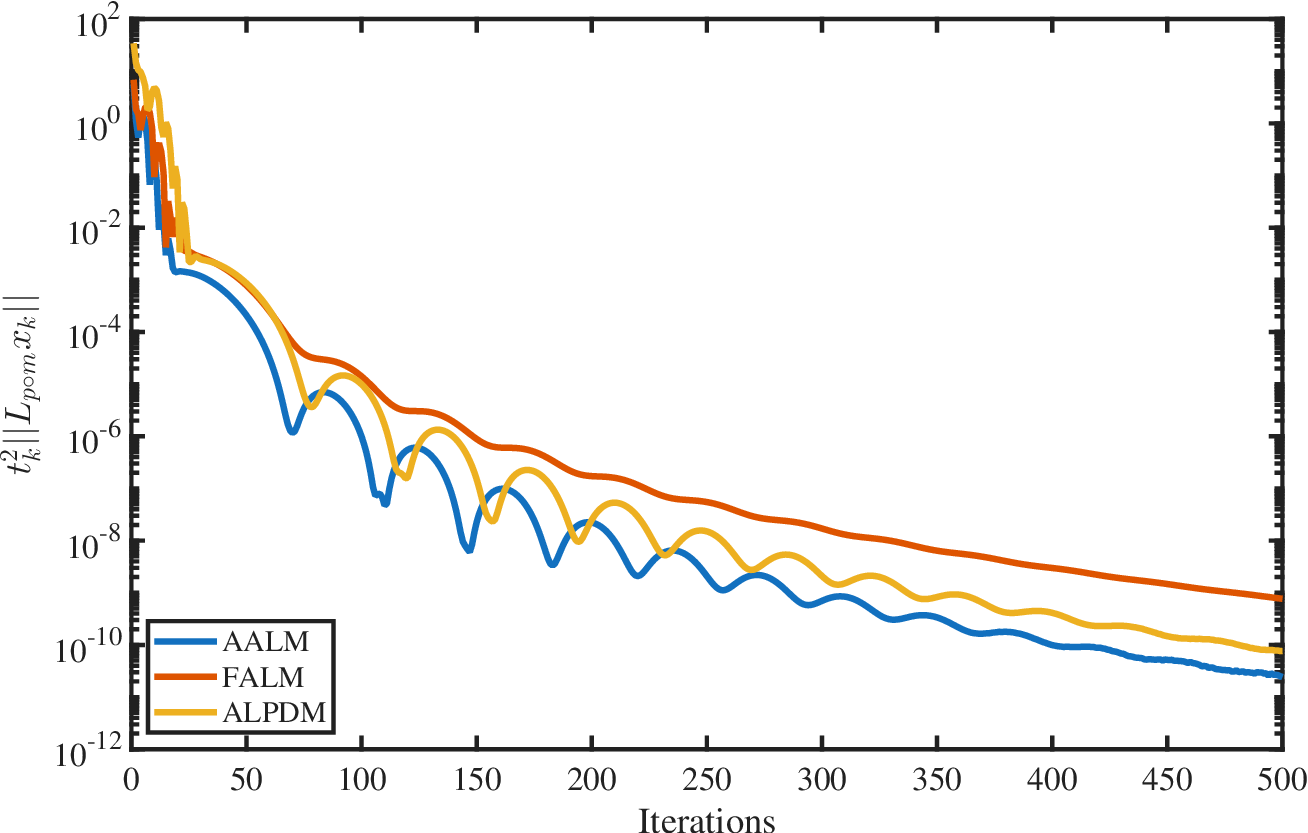} &
        \includegraphics[width=0.31\textwidth]{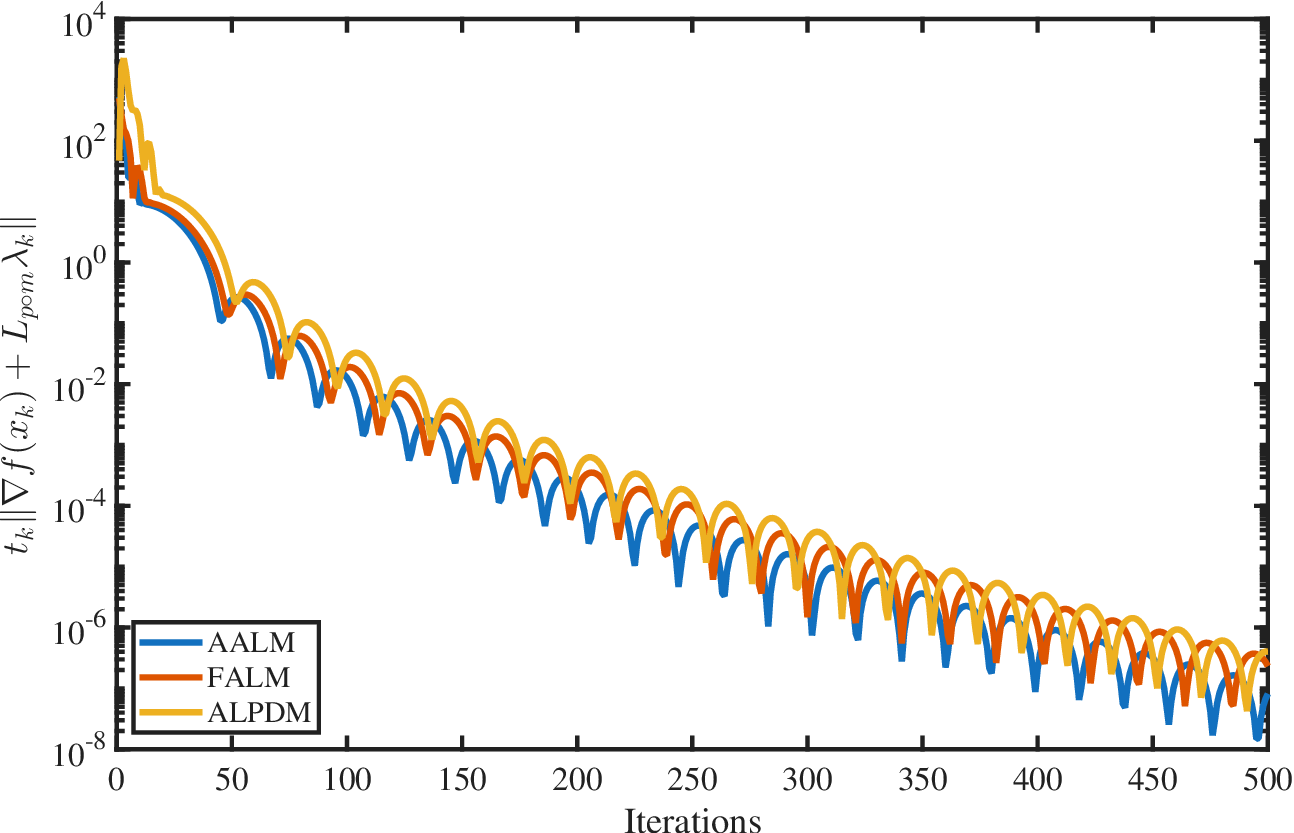}
    \end{tabular}
    \caption{Numerical comparison of AALM, FALM and ALPDM for solving
    \eqref{ques_X1}.}
    \label{fig1}
\end{figure}

  \section{Conclusion}\label{sec6}

In this paper, we proposed an accelerated augmented Lagrangian method for the
linearly constrained convex problem \eqref{prob_main}. The method applies to
convex $C^1$ objectives and also admits a partially explicit variant for convex
$L$-smooth objectives. Based on a unified discrete Lyapunov estimate, we proved accelerated convergence rates for the augmented Lagrangian gap,
feasibility violation and objective residual, and convergence of the  
primal-dual sequence. In the noncritical regime, the standard $\bigO(1/t_k^2)$
bounds were improved to $o(1/t_k^2)$ rates. In the smooth case, we further
obtained an $o(1/t_k)$ estimate and weighted square summability for the
stationarity residual. The numerical results are consistent with the established
theoretical convergence properties.

\appendix

\section{Auxiliary results}
\setcounter{equation}{0}
\renewcommand{\theequation}{A.\arabic{equation}}
\renewcommand{\thesection}{\Alph{section}}

In this appendix, we collect several auxiliary results used in the convergence
analysis of the proposed Algorithm. We first recall a basic property of the KKT set.

\begin{proposition}\cite[Proposition A.1]{HeArxiv}\label{prop_App1}
Consider the linearly constrained problem \eqref{prob_main}, where
$f:\R^n\to\R$ is convex and differentiable. Then, for any
$(x_1^*,\lambda_1^*), (x_2^*,\lambda_2^*)\in \Omega$, one has
\[
\nabla f(x_1^*)=\nabla f(x_2^*)\quad \text{and} \quad
A^\top\lambda_1^*=A^\top\lambda_2^*.
\]
\end{proposition}

The next estimate is a simple consequence of the descent lemma \cite[Theorem 2.1.5]{Nesterov2004}. We shall use it to handle the explicit
gradient evaluation in the smooth case.

\begin{lemma}\cite[Lemma 3]{HeCnsnsPD}\label{le_eqsmf}
If $f:\mathbb{R}^{n}\to\mathbb{R}$ is convex and has a Lipschitz continuous
gradient with constant $L$, then, for any $x,y,z\in\mathbb{R}^{n}$, one has
\begin{eqnarray*}
   \langle \nabla f(z), x-y\rangle\geq
   f(x)-f(y)-\frac{L}{2}\|x-z\|^2.
\end{eqnarray*}
\end{lemma}

We next give a boundedness lemma for inertial sequences.  
\begin{lemma}\label{le_wbound}
Let $M\in\mathbb R^{N\times N}$ be symmetric positive definite, and let
$\eta>0$. Let $\{a_k\}_{k\ge1}$ be a nonnegative sequence, and let
$\{z_k\}_{k\ge 0}$ be a sequence in $\mathbb R^N$ with $z_1=z_0$.
Fix $z^*\in \mathbb R^N$ and define
\[
    w_k:=\eta(z_k-z^*)+a_k(z_k-z_{k-1}),
    \qquad k\ge1.
\]
Assume that $W:=\sup_{k\ge1}\|w_k\|_M<+\infty$. Then
\[
    \|z_k-z^*\|_M\le \frac{W}{\eta},\quad
    {a_k}\|z_k-z_{k-1}\|_M\le{2W},
    \qquad \forall k\ge 1.
\]
\end{lemma}

\begin{proof}
Set $y_k=z_k-z^*$. Then
\[
    w_k=\eta y_k+a_k(y_k-y_{k-1})
    =
    (\eta+a_k)y_k-a_k y_{k-1}.
\]
Let $\theta_k=\frac{a_k}{\eta+a_k}$. It follows that
\[
    y_k
    =
    \theta_k y_{k-1}
    +
    (1-\theta_k)\frac{w_k}{\eta}.
\]
Since $M$ is symmetric positive definite and
$W=\sup_{k\ge1}\|w_k\|_M<+\infty$, we obtain
\begin{equation}\label{eq_ykM}
     \|y_k\|_M
    \le
    \theta_k\|y_{k-1}\|_M
    +
    (1-\theta_k)\frac{W}{\eta}.
\end{equation}
Since $z_1=z_0$, we have $w_1=\eta(z_1-z^*)=\eta y_1$, and hence
$
    \|y_1\|_M\le \frac{W}{\eta}.
$
Therefore, by induction and \eqref{eq_ykM},
\begin{equation}\label{eq_zbound}
    \|z_k-z^*\|_M=\|y_k\|_M\le \frac{W}{\eta},
    \qquad \forall k\ge1.
\end{equation}

Next, from the definition of $w_k$,
\[
    a_k(z_k-z_{k-1})
    =
    w_k-\eta(z_k-z^*).
\]
Using \eqref{eq_zbound} and the definition of $W$, we get
\[
a_k\|z_k-z_{k-1}\|_M
\le
\|w_k\|_M+\eta\|z_k-z^*\|_M
\le 2W.
\]
This proves the second estimate and completes the proof.
\end{proof}

The following elementary sequence estimate will be used to control the scaled
feasibility residual.

\begin{lemma}\cite[Lemma 6]{HeAutomatica}\label{le_heauto}
Let $\{g_k\}_{k\geq 1}$ be a sequence of vectors in $\mathbb{R}^{n}$ and let
$\{a_k\}_{k\geq 1}\subset[0,1)$. Assume that there exists $C>0$ such that
\[
    \left\|g_{k+1}+\sum_{j=1}^ka_j g_j\right\|\leq C,
    \qquad \forall k\geq 1.
\]
Then
\[
    \sup_{k\geq 1}\|g_k\|\leq \|g_1\|+2C.
\]
\end{lemma}

We shall also use the following standard deterministic convergence lemma.

\begin{lemma}\cite[Proposition~A.30]{Bertsekas} \label{le_bertsekas}
Let $\{\omega_k\}_{k\ge1}$, $\{\sigma_k\}_{k\ge1}$, and $\{q_k\}_{k\ge1}$
be nonnegative sequences satisfying
\[
    \omega_{k+1}\le (1-\sigma_k)\omega_k+\sigma_k q_k,
    \qquad 0< \sigma_k\le 1,
    \qquad \sum_{k=1}^{+\infty}\sigma_k=+\infty.
\]
If $q_k\to0$ as $k\to+\infty$, then
$
    \omega_k\to0.
$
\end{lemma}

The next averaging lemma is used to prove the little-o estimate of the
feasibility violation.

\begin{lemma}\label{lem_F_res}
Let $\{r_k\}_{k\ge1}$ be a sequence in $\mathbb R^N$, and let
$\{a_k\}_{k\ge1}\subset(0,1)$ satisfy
$
    \sum_{k=1}^{+\infty}a_k=+\infty.
$
Assume that
\[
    F_{k+1}:=r_{k+1}+\sum_{j=1}^{k}a_j r_j
\]
has a finite limit in $\mathbb R^N$ as $k\to+\infty$. Then
\[
    \lim_{k\to+\infty}\|r_k\| = 0.
\]
\end{lemma}

\begin{proof}
Set $S_k:=\sum_{j=1}^{k}a_j r_j$. Then $F_{k+1}=r_{k+1}+S_k$ and
\[
    S_{k+1}
    =
    S_k+a_{k+1}r_{k+1}
    =
    (1-a_{k+1})S_k+a_{k+1}F_{k+1}.
\]
Let $F_k\to F_\infty\in\mathbb R^N$. Then
\[
 \|S_{k+1}-F_\infty\|
 \le
 (1-a_{k+1})\|S_k-F_\infty\|
 +
 a_{k+1}\|F_{k+1}-F_\infty\|.
\]
Applying Lemma \ref{le_bertsekas} with
$\omega_k=\|S_k-F_\infty\|$ yields
\[
    S_k\to F_\infty.
\]
Consequently,
\[
    \lim_{k\to+\infty}\|r_{k+1}\|
    =
    \lim_{k\to+\infty}\|F_{k+1}-S_k\|
    =0.
\]
\end{proof}

The last lemma  ensures the divergence of $\sum_k1/t_k$, which is
needed in the averaging arguments.

\begin{lemma}\label{lem_scaledtk}
Let $\{t_k\}_{k\ge1}$ and $\{\xi_k\}_{k\ge1}$ be two positive sequences in
$[1,+\infty)$. Assume that $t_k^2\xi_k$ is nondecreasing,
$t_k^2\xi_k\to+\infty$, and
\[
    t_{k+1}^2\xi_{k+1}-t_k^2\xi_k
    \le
    \rho t_{k+1}\xi_{k+1},
    \qquad
    0<\rho\le1.
\]
Then
$
    \sum_{k=1}^{+\infty}\frac1{t_k}=+\infty.
$
\end{lemma}

\begin{proof}
Suppose, by contradiction, that
$
    \sum_{k=1}^{+\infty}\frac1{t_k}<+\infty.
$
Then $1/t_k\to0$. Hence there exists $k_0\ge1$ such that
\begin{equation}\label{eq_tk12}
      \frac{\rho}{t_{k+1}}\le \frac12,
    \qquad \forall k\ge k_0.
\end{equation}
From the scaled growth condition,
\[
    t_{k+1}^2\xi_{k+1}
    -
    t_k^2\xi_k
    \le
    \frac{\rho}{t_{k+1}}t_{k+1}^2\xi_{k+1}.
\]
Thus, for all $k\ge k_0$,
\[
    t_{k+1}^2\xi_{k+1}
    \le
    \left(1-\frac{\rho}{t_{k+1}}\right)^{-1}
    t_k^2\xi_k.
\]
Iterating from $k_0$ to $N$ gives
\begin{equation}\label{eq_app00}
      t_{N+1}^2\xi_{N+1}
    \le
    t_{k_0}^2\xi_{k_0}
    \prod_{k=k_0}^{N}
    \left(1-\frac{\rho}{t_{k+1}}\right)^{-1}.
\end{equation}

We now estimate the product. For $0\le s\le 1/2$, define
$\phi(s):=2s+\ln(1-s)$. Then
\[
    \phi'(s)=2-\frac1{1-s}
    =
    \frac{1-2s}{1-s}\ge0.
\]
Since $\phi(0)=0$, we have $\phi(s)\ge0$, and hence
$
    -\ln(1-s)\le 2s.
$
Therefore,
\[
    (1-s)^{-1}\le e^{2s},
    \qquad 0\le s\le\frac12.
\]
Using \eqref{eq_tk12}, we get
\[
    \left(1-\frac{\rho}{t_{k+1}}\right)^{-1}
    \le
    \exp\left(\frac{2\rho}{t_{k+1}}\right),
    \qquad k\ge k_0.
\]
Combining this estimate with \eqref{eq_app00}, we obtain
\[
    t_{N+1}^2\xi_{N+1}
    \le
    t_{k_0}^2\xi_{k_0}
    \exp\left(
    2\rho\sum_{k=k_0}^{N}\frac1{t_{k+1}}
    \right).
\]
Since
$
    \sum_{k=1}^{+\infty}\frac1{t_k}<+\infty,
$
the right-hand side is uniformly bounded in $N$. This contradicts
$t_k^2\xi_k\to+\infty$. Therefore,
\[
    \sum_{k=1}^{+\infty}\frac1{t_k}=+\infty.
\]
\end{proof}

\end{document}